\documentclass[a4paper]{cas-sc}

\usepackage{amsmath,amssymb,amsthm,mathtools,mathrsfs}
\usepackage[british]{babel}
\usepackage[T1]{fontenc}
\usepackage[utf8]{inputenc}
\usepackage{xcolor}

\usepackage[numbers,sort&compress]{natbib}
\allowdisplaybreaks

\newcommand{\R}{\mathbb R}
\newcommand{\Torus}{\mathbb T}
\newcommand{\eps}{\varepsilon}
\newcommand{\dt}{\partial_t}
\newcommand{\diver}{\operatorname{div}}
\newcommand{\LL}{\mathbf L}
\newcommand{\Sbulk}{s_{\rm bulk}}
\newcommand{\Stot}{\mathscr S}
\newcommand{\Energy}{\mathscr E}
\newcommand{\Avail}{\mathscr A}
\newcommand{\Diss}{\mathscr D}
\newcommand{\Vh}{V_h}
\newcommand{\Ih}{I_h}

\newcommand{\deltat}{\delta_\tau}
\newcommand{\pair}[2]{\left\langle #1,#2\right\rangle}
\newcommand{\norm}[1]{\left\lVert #1\right\rVert}

\newcommand{\dx}{\,\mathrm dx}
\newcommand{\dtm}{\,\mathrm dt}
\newcommand{\QT}{(0,T)\times\Omega}
\newcommand{\vex}{\mathrm{vex}}
\newcommand{\cav}{\mathrm{cav}}
\newcommand{\sing}{\mathrm{sing}}
\newcommand{\reg}{\mathrm{reg}}

\newtheorem{assumption}{Assumption}[section]
\newtheorem{theorem}[assumption]{Theorem}
\newtheorem{lemma}[assumption]{Lemma}
\newtheorem{proposition}[assumption]{Proposition}
\newtheorem{corollary}[assumption]{Corollary}
\theoremstyle{definition}
\newtheorem{definition}[assumption]{Definition}
\newtheorem{remark}[assumption]{Remark}

\def\aaron#1{{\color{purple} #1}}

\begin{document}
\let\WriteBookmarks\relax
\def\floatpagepagefraction{1}
\def\textpagefraction{.001}

\shorttitle{Generalised dissipative solutions for a non-isothermal phase-field system}
\shortauthors{A. Brunk and M. Fritz}

\title[mode=title]{Generalised dissipative solutions for a non-isothermal phase-field system: existence, weak--strong uniqueness, and long-time behaviour}

\author[1]{Aaron Brunk}[orcid=0000-0003-4987-2398]
\ead{abrunk@uni-mainz.de}
\affiliation[1]{organization={Institute of Mathematics, Johannes Gutenberg University Mainz},
                 city={Mainz},
                 country={Germany}}
\cortext[1]{Corresponding author}
\author[2]{Marvin Fritz}[orcid=0000-0002-8360-7371]
\ead{marvin.fritz@univie.ac.at}
\affiliation[2]{organization={Faculty of Mathematics, University of Vienna},
                 city={Vienna},
                 country={Austria}}


\begin{abstract}
We study a thermodynamically consistent non-isothermal phase-field system coupling two order parameters and the inverse temperature through a fully non-diagonal Onsager mobility. The model describes the interaction of mass diffusion, heat conduction, and local phase relaxation while conserving mass and internal energy and producing entropy. Global generalised dissipative weak solutions are constructed using a fully discrete approximation. The discrete scheme conserves mass and internal energy and satisfies a discrete entropy inequality. An availability estimate, together with a dimension-adapted barrier, yields strict positivity at fixed mesh and uniform estimates. Compactness then allow passage to the continuous system. Concentration of the singular part of the internal energy is represented by a non-negative defect measure. The entropy–availability structure yields finite dissipation on the infinite time interval and the existence of stationary $\omega$-limit states. Finally, a relative-entropy argument establishes weak–strong uniqueness whenever the weak and strong solutions remain in a bounded thermodynamic state range.
\end{abstract}


\maketitle

\section{Introduction}\label{sec:intro}

Phase-field models provide a versatile framework for interfacial phenomena
such as solidification, phase separation, microstructure evolution, tumour
growth, fracture, and multiphase transformations; see, for example,
\cite{fritz2023tumor}. The Cahn--Hilliard equation
\cite{cahn1958free} is a fundamental model for isothermal phase
transitions. In many applications, however, heat conduction, latent heat
exchange, and temperature-dependent interactions cannot be neglected, and
the phase evolution must be coupled consistently to the balances of mass
and internal energy and to the second law of thermodynamics.

A systematic thermodynamic derivation of non-isothermal phase-field models
was developed by Penrose and Fife \cite{PenroseFife1990} and by Alt and
Pawlow
\cite{alt1990dynamics,alt1991mathematical,alt1992mathematical,
alt1991existence,alt1995thermodynamical}. Their Lyapunov-functional
analysis also relates finite entropy production to stationary states and,
near isolated equilibria, to nonlinear asymptotic stability; see
\cite[Proposition~2.6]{alt1991mathematical} and
\cite[Section~3.4]{alt1992mathematical}. Weak solutions for
Penrose--Fife systems and related models have subsequently been studied
under a variety of constitutive assumptions in
\cite{Colli1998,MiranvilleRoccaSchimperna2013,Marveggio2021,
liu2025nonisothermal,cavalleri2025well}. Stabilisation, optimal control, and
temperature-control problems are considered in
\cite{AzmiFritzRod24-arx,colli2023optimal,colli2023optimalb,
giulia2026optimal,cavalleri2025temperature}. Systems with several order
parameters arise, for instance, in models of sintering, grain coalescence,
and thermomechanical phase transformations
\cite{Oyedeji2023,maraldi2011phase}, while non-isothermal
Cahn--Hilliard systems coupled to fluid flow have been investigated in
\cite{deliyianni2026temperature}. Structure-preserving approximations are
developed in
\cite{kahnt2021numerical,brunk2025structure,brunk2026structure}, and
analytical techniques for diffuse-interface systems with cross-diffusion
and non-diagonal kinetic structures can be found in
\cite{alt1983quasilinear,abels2009existence,
BrunkEggerOyedejiYangXu2022,HuoJuengelTzavaras2022}. Across these settings,
positivity of the temperature, conservation of internal energy, and entropy
production are central structural requirements.

Beyond existence, weak--strong uniqueness gives a natural stability
principle for dissipative systems. The relative entropy method has proved
effective for thermodynamically consistent fluid models
\cite{FeireislWeakStrong2012} and has also been adapted to
diffuse-interface systems
\cite{LasarzikRoccaSchimperna2019,brunk2023existence}. In the present
setting, the corresponding relative functional is expressed through the
availability and is compatible with the system's entropy variables.

We consider a non-isothermal phase-field system for two coupled order
parameters $\rho$ and $\eta$ and the inverse temperature $\theta=1/T$.
The variable $\rho$ follows a conserved Cahn--Hilliard-type evolution,
whereas $\eta$ is non-conserved. Their chemical potentials, the heat flux,
and the local relaxation mechanism are coupled through a symmetric,
uniformly positive definite Onsager operator whose off-diagonal entries are
retained. The model conserves total mass and internal energy and produces
entropy through a positive quadratic dissipation.

Our existence proof is based on an implicit fully discrete approximation
with continuous piecewise affine finite elements developed in \cite{BrunkHabrichOyedejiYangXu}. The phase-dependent part
of the reduced free energy is treated by a convex--concave splitting, while
all temporal and nonlinear thermodynamic terms use the consistent
$L^2$-pairing. Thus neither mass lumping nor nodal quadrature is needed.
The discrete scheme conserves exactly global mass and internal
energy and satisfies a discrete entropy inequality; see
Theorem~\ref{thm:discrete-structure}.

The main difficulty is strict positivity of the discrete inverse
temperature. We introduce a dimension-adapted singular contribution to the
reduced free energy whose associated inverse moment controls the approach to
zero. For piecewise affine functions, the condition $q\ge d$ yields a
fixed-mesh nodal barrier and is sharp for this argument; in three dimensions
the minimal choice is $q=3$. Combined with a coercive availability estimate,
this barrier permits a Brouwer degree construction of a positive solution at
every time step; see
Lemma~\ref{lem:fe-positivity-barrier} and
Theorem~\ref{thm:one-step-existence}. The availability also controls the
spatial mean of the inverse temperature, so a generalised Poincar\'e
inequality upgrades the entropy dissipation to a full
$L^2(0,T;H^1(\Omega))$ estimate.

The uniform estimates yield compactness of the phase variables by a
discrete Aubin--Lions argument. For the inverse temperature, the strict
monotonicity of the internal energy and the logarithmic thermal term provide
time-translation control, while the low- and high-temperature estimates and
a Poincar\'e inequality on subsets of positive measure give strong
convergence in $L^2(\QT)$. Consequently, for arbitrary
$h,\tau\to0$, without a coupling condition, a subsequence of the discrete
solutions converges to a global generalised dissipative weak solution; see
Theorem~\ref{thm:convergence}. All lower-order thermal nonlinearities are
identified strongly. Possible concentration of the critical singular
internal-energy term is represented by a non-negative Radon measure in the
limiting energy balance, and sufficient conditions for its elimination are
given in Proposition~\ref{prop:no-defect}.

We further establish weak--strong uniqueness by a relative-entropy
argument. On a common bounded thermodynamic state range, the relative
functional is locally coercive and its evolution controls the full
non-diagonal Onsager dissipation, the interfacial terms, and the possible
energy defect. A Gronwall argument shows that a generalised dissipative weak
solution coincides with a sufficiently regular strong solution with the
same initial data; the concentration defect then vanishes. This is stated
in Theorem~\ref{thm:bounded-range-wsu}. 

Finally, for global trajectories, the same entropy--availability inequality gives
finite total Onsager dissipation.  We prove that every global generalised
dissipative weak solution possesses a sequence of times tending to infinity
at which both phase variables and the inverse temperature converge strongly
in $H^1$ to a stationary state.  The critical inverse power may develop an
additional concentration along this long-time sequence; it is retained as a
non-negative asymptotic energy defect.  Thus the result identifies
stationary $\omega$-limit states without imposing the isolation assumptions
needed for convergence of the entire trajectory; see
Theorem~\ref{thm:stationary-omega-limit}.

The remainder of the paper is organised as follows.
Section~\ref{sec:thermodynamics} presents the thermodynamic formulation, the
dimension-adapted constitutive assumptions, and the notion of a generalised
dissipative weak solution.  The main existence and convergence theorem is
stated at the beginning of Section~\ref{sec:existence}; its proof is then
developed through the discrete scheme, the availability estimates,
fixed-mesh solvability, compactness, and the identification of the limiting
energy defect.  Section~\ref{sec:wsu} contains the relative-entropy analysis,
the weak--strong uniqueness theorem, and the long-time identification of
stationary $\omega$-limit states.

\section{Thermodynamic structure and constitutive assumptions}
\label{sec:thermodynamics}

\subsection{Thermodynamic formulation}
\label{sec:model}

Let $\Omega=\Torus^d$, $1\le d\le3$, and let
$$
 \rho=\rho(t,x),\qquad \eta=\eta(t,x),\qquad
 \theta=\theta(t,x)=T(t,x)^{-1}>0.
$$
We write $z=(\rho,\eta)$.  The reduced bulk Helmholtz free-energy density is
$\psi=\psi(z,\theta)$, and the reduced free energy including interfacial
terms is
\begin{equation}
 \Psi(z,\theta,\nabla z)
 =\psi(z,\theta)
 +\frac{\gamma_\rho}{2}|\nabla\rho|^2
 +\frac{\gamma_\eta}{2}|\nabla\eta|^2,
 \qquad \gamma_\rho,\gamma_\eta>0.
 \label{eq:Psi}
\end{equation}
The word {reduced} is essential: in the inverse-temperature variable
we use
\begin{equation}
 e(z,\theta)=\partial_\theta\psi(z,\theta),
 \qquad
 \Sbulk(z,\theta)=\theta e(z,\theta)-\psi(z,\theta).
 \label{eq:thermo}
\end{equation}
The total entropy functional is
\begin{equation}
 \Stot(z,\theta)
 =\int_\Omega\Sbulk(z,\theta)\dx
 -\frac{\gamma_\rho}{2}\norm{\nabla\rho}_{L^2}^2
 -\frac{\gamma_\eta}{2}\norm{\nabla\eta}_{L^2}^2.
 \label{eq:total-entropy}
\end{equation}
The chemical potentials are
\begin{equation}
 \mu_\rho=-\gamma_\rho\Delta\rho+\partial_\rho\psi(z,\theta),
 \qquad
 \mu_\eta=-\gamma_\eta\Delta\eta+\partial_\eta\psi(z,\theta).
 \label{eq:chemical-potentials}
\end{equation}
Since $e=\partial_\theta\psi$, differentiation of
$\Sbulk=\theta e-\psi$ gives, for every variation,
$$
 D\Sbulk(z,\theta)[\delta z,\delta\theta]
 =\theta\,De(z,\theta)[\delta z,\delta\theta]
  -D_z\psi(z,\theta)\cdot\delta z.
$$
After integration over the torus, the variations of the two interfacial
terms are integrated by parts.  With
\eqref{eq:chemical-potentials} this yields the Gibbs relation
\begin{equation}
 D\Stot(z,\theta)[\delta z,\delta\theta]
 =(\theta,\delta e)-(\mu_\rho,\delta\rho)
 -(\mu_\eta,\delta\eta).
 \label{eq:gibbs}
\end{equation}

The system then reads
\begin{subequations}
\label{eq:model-system}
\begin{align}
 \dt\rho
 &=\diver\bigl(\LL_{11}\nabla\mu_\rho
       -\LL_{12}\nabla\theta+\LL_{13}\mu_\eta\bigr),
 \label{eq:model-rho}\\
 \dt e(z,\theta)
 &=\diver\bigl(\LL_{12}^{\top}\nabla\mu_\rho
       -\LL_{22}\nabla\theta+\LL_{23}\mu_\eta\bigr),
 \label{eq:model-energy}\\
 \dt\eta
 &=-\LL_{13}\cdot\nabla\mu_\rho
       +\LL_{23}\cdot\nabla\theta-\LL_{33}\mu_\eta,
 \label{eq:model-eta}
\end{align}
\end{subequations}
completed by \eqref{eq:chemical-potentials}, where the mobility is assumed to admit the block form
\begin{equation}
 \LL(z,\theta)=
 \begin{pmatrix}
  \LL_{11}&\LL_{12}&\LL_{13}\\
  \LL_{12}^{\top}&\LL_{22}&\LL_{23}\\
  \LL_{13}^{\top}&\LL_{23}^{\top}&\LL_{33}
 \end{pmatrix}.
 \label{eq:L}
\end{equation}
Here $\LL_{11},\LL_{12},\LL_{22}\in\R^{d\times d}$,
$\LL_{13},\LL_{23}\in\R^d$, and $\LL_{33}\in\R$. 

If
$
 Z=(\nabla\mu_\rho,-\nabla\theta,\mu_\eta)$,
then smooth periodic solutions satisfy
\begin{equation}
 \frac{\mathrm d}{\mathrm dt}\int_\Omega\rho\dx=0,
 \qquad
 \frac{\mathrm d}{\mathrm dt}\int_\Omega e(z,\theta)\dx=0,
 \qquad
 \frac{\mathrm d}{\mathrm dt}\Stot(z,\theta)
 =\int_\Omega Z^\top\LL Z\dx\ge0.
 \label{eq:formal-structure}
\end{equation}
Indeed, the first two identities in \eqref{eq:formal-structure} follow by integrating
\eqref{eq:model-rho} and \eqref{eq:model-energy} over the periodic domain.
For the entropy, use \eqref{eq:gibbs} with the time derivatives and insert
the three balance laws.  Periodic integration by parts gives
\begin{align*}
 \frac{\mathrm d}{\mathrm dt}\Stot
 &=(\theta,\partial_t e)
   -(\mu_\rho,\partial_t\rho)
   -(\mu_\eta,\partial_t\eta)=(-\nabla\theta,\mathcal J_e)
   +(\nabla\mu_\rho,\mathcal J_\rho)
   +(\mu_\eta,\mathcal R_\eta)
 =\int_\Omega Z^\top\LL Z\dx,
\end{align*}
where $\mathcal J_\rho$, $\mathcal J_e$, and $\mathcal R_\eta$ denote the
three Onsager expressions in \eqref{eq:model-system}.  The last quantity is
non-negative by the following assumption:

\begin{assumption}[bounded Onsager matrix]\label{ass:Onsager}
We assume throughout
that $\LL$ is continuous, bounded, symmetric, and uniformly positive
definite: there exist $0<\lambda_0\le\lambda_1$ such that
\begin{equation}
 \lambda_0|\zeta|^2\le \zeta^\top\LL(z,\theta)\zeta
 \le\lambda_1|\zeta|^2
 \qquad
 \text{for every }\zeta\in\R^{2d+1}.
 \label{eq:L-ellipticity}
\end{equation}
\end{assumption}


\subsection{Dimension-adapted thermal barrier and coercive availability}
\label{sec:potential}

The thermal modification used below follows a classical regularisation
principle from the analysis of heat-conducting fluids. Singular
inverse-temperature contributions are introduced at the approximation level
to prevent degeneration of the absolute temperature, while radiative terms
provide superlinear control in the high-temperature regime; see
\cite[Sections~1.4.3, 3.3.1, and 3.4.2]{FeireislNovotny2017} and
\cite{Feireisl2004Heat,DucometFeireisl2005}. In terms of the absolute
temperature $T=\theta^{-1}$, the contribution used here generates energy
and entropy terms proportional to $T^q$ and $T^{q-1}$, respectively; for
$q=4$ this is the standard radiative pair. Our choice $q\ge d$ is adapted
to the finite-element positivity argument: it is exactly the threshold at
which the inverse moment excludes vanishing at a vertex of a positive
piecewise affine function.

\begin{assumption}[Modified reduced free energy]
\label{ass:potential}
Let $q>1$ satisfy
\begin{equation}
 q\ge d,
 \label{eq:q-dimension}
\end{equation}
let $\eps_1,\eps_q>0$, and let $F_0,F_1$ be non-negative polynomials on
$\R^2$ of degree at most four.  We set
\begin{equation}
 \psi(z,\theta)
 =\ln\theta-\eps_1\theta(\ln\theta-1)
 -\frac{\eps_q}{q-1}\theta^{1-q}
 -F_0(z)+\theta F_1(z),
 \qquad \theta>0.
 \label{eq:modified-potential}
\end{equation}
We assume
\begin{equation}
 |D F_j(z)|\le C(1+|z|^3),
 \qquad
 |D^2F_j(z)|\le C(1+|z|^2),
 \qquad j=0,1,
 \label{eq:phase-growth}
\end{equation}
and that there is a number $a>0$ for which
\begin{equation}
 aF_1(z)-F_0(z)\ge c_a|z|^4-C_a.
 \label{eq:phase-availability-coercivity}
\end{equation}
\end{assumption}

\begin{remark}[Scope of the thermal modification]
\label{rem:fixed-thermal-parameters}
\phantom{x}
\begin{itemize}
    \item The parameters $\eps_1$ and $\eps_q$ are fixed positive constitutive
parameters throughout.  Thus all existence, convergence, and stability
results below concern the modified model \eqref{eq:modified-potential}.  No
vanishing-regularisation limit $\eps_1,\eps_q\downarrow0$ is asserted; such
a limit would require estimates uniform in these parameters and a separate
identification of the singular thermal terms.
\item Instead of modifying the free energy, one may alternatively obtain additional thermal control from a heat conductivity with suitable power-law growth in the temperature; see, for example, \cite{Michela}.
\end{itemize}
\end{remark}

The minimal dimension-adapted choice in three dimensions is $q=3$, which
adds $-\eps_3(2\theta^2)^{-1}$ to $\psi$ and
$\eps_3\theta^{-3}$ to $e$.  The concrete phase energy used in \cite{Oyedeji2023,BrunkHabrichOyedejiYangXu} also fits Assumption~\ref{ass:potential}.  Namely, let
\begin{equation}
 F(\rho)=\rho^2(1-\rho)^2,
 \qquad
 Q(\rho,\eta)=(\rho-1)^2+12\eta^2(1-\eta)^2,
 \label{eq:FQ}
\end{equation}
and take
\begin{equation}
 F_1=C_1F+D_1Q,
 \qquad
 F_0=C_2F+D_2Q,
 \qquad C_i,D_i>0.
 \label{eq:F0F1}
\end{equation}
Then \eqref{eq:phase-availability-coercivity} holds for every
\begin{equation}
 a>\max\left\{\frac{C_2}{C_1},\frac{D_2}{D_1}\right\}.
 \label{eq:a-choice}
\end{equation}
Indeed, both coefficients in
$
 aF_1-F_0=(aC_1-C_2)F+(aD_1-D_2)Q
$
are then positive, and the explicit polynomials satisfy
$F(\rho)+Q(\rho,\eta)\ge c(|\rho|^4+|\eta|^4)-C$.

\begin{lemma}[Thermodynamic identities]
\label{lem:thermo-identities}
Under Assumption~\ref{ass:potential},
\begin{align}
 e(z,\theta)
 &=\frac1\theta-\eps_1\ln\theta
   +\eps_q\theta^{-q}+F_1(z),
 \label{eq:energy-density}\\
 \Sbulk(z,\theta)
 &=1-\ln\theta-\eps_1\theta
   +\frac{q}{q-1}\eps_q\theta^{1-q}+F_0(z),
 \label{eq:entropy-density}\\
 \partial_\theta e(z,\theta)
 &=\partial_{\theta\theta}\psi(z,\theta)
 =-\frac1{\theta^2}-\frac{\eps_1}{\theta}
   -q\eps_q\theta^{-q-1}<0.
 \label{eq:energy-monotonicity}
\end{align}
Thus $\psi$ is strictly concave in the inverse temperature and, for fixed
$z$, the map $\theta\mapsto e(z,\theta)$ is strictly decreasing.
\end{lemma}

\begin{proof}
Differentiate \eqref{eq:modified-potential} and use
$\Sbulk=\theta e-\psi$.
\end{proof}

For the fixed number $a$ in
\eqref{eq:phase-availability-coercivity}, define the availability and its bulk density via
\begin{align}
 \Avail_a(z,\theta)
 &:=a\int_\Omega e(z,\theta)\dx-\Stot(z,\theta).
 \label{eq:availability-def}\\
 a e-\Sbulk
 ={}\eps_1\theta+a\eps_q\theta^{-q}
 +a\theta^{-1}-&\frac{q}{q-1}\eps_q\theta^{1-q}
 +(1-a\eps_1)\ln\theta-1+aF_1(z)-F_0(z).
 \label{eq:availability-density}
\end{align}

\begin{lemma}[Coercivity of the availability]
\label{lem:availability-coercivity}
There are constants $c,C>0$, depending only on the constitutive
coefficients and on $a$, such that
\begin{equation}
 a e(z,\theta)-\Sbulk(z,\theta)
 \ge c\bigl(\theta+\theta^{-q}+|z|^4\bigr)-C
 \qquad(z\in\R^2,\ \theta>0).
 \label{eq:availability-coercive}
\end{equation}
Moreover,
\begin{equation}
 |e(z,\theta)|+|\Sbulk(z,\theta)|
 \le C\bigl(1+\theta+\theta^{-q}+|z|^4\bigr).
 \label{eq:e-s-growth}
\end{equation}
\end{lemma}

\begin{proof}
For every $\delta>0$ and $0<\theta\le1$,
$$
 \theta^{1-q}+\theta^{-1}+|\ln\theta|
 \le \delta\theta^{-q}+C_\delta.
$$
Choose $\delta$ small enough to absorb every possibly negative lower-order
term into $a\eps_q\theta^{-q}$.  For $\theta\ge1$, the term
$\eps_1\theta$ absorbs $|\ln\theta|$, while all negative powers are
bounded.  Combining these observations with
\eqref{eq:phase-availability-coercivity} proves
\eqref{eq:availability-coercive}.  The same elementary comparisons and
\eqref{eq:phase-growth} give \eqref{eq:e-s-growth}.
\end{proof}

For the convex--concave splitting, set
$$
 H(z)=(1+|z|^2)^2.
$$

\begin{lemma}[Global phase-variable splitting]
\label{lem:global-splitting}
There is $M>0$, depending only on the constants in
\eqref{eq:phase-growth}, such that
\begin{equation}
 \psi_\vex(z,\theta)=\psi(z,\theta)+M(1+\theta)H(z),
 \qquad
 \psi_\cav(z,\theta)=-M(1+\theta)H(z)
 \label{eq:global-splitting}
\end{equation}
has the following properties for every $\theta>0$:
$\psi_\vex(\cdot,\theta)$ is convex,
$\psi_\cav(\cdot,\theta)$ is concave, and
\begin{equation}
 |D_z\psi_\sigma(z,\theta)|
 \le C(1+\theta)(1+|z|^3),
 \qquad \sigma\in\{\vex,\cav\}.
 \label{eq:split-growth}
\end{equation}
\end{lemma}

\begin{proof}
The Hessian of $H$ is
$$
 D^2H(z)=4(1+|z|^2)I+8z\otimes z
 \ge4(1+|z|^2)I.
$$
Moreover,
$$
 D_z^2\psi(z,\theta)
 =-D^2F_0(z)+\theta D^2F_1(z)
 \ge-C(1+\theta)(1+|z|^2)I.
$$
Choose $M\ge C/4$.  The derivative estimate follows from the cubic growth
of $DF_0$, $DF_1$, and $DH$.
\end{proof}

\begin{remark}
 The availability is a rescaled exergy.  Indeed, for a fixed  reference state $(e_0,s_0,T_0)$, the exergy is
$W=(e-e_0)-T_0(s-s_0)$.  Terms depending only on the reference state are irrelevant for the evolution, and division by $T_0$ gives
$W/T_0=\theta_0 e-s$ up to an additive constant, where
$\theta_0=T_0^{-1}$.  Thus $\Avail_a$ corresponds to a fixed reference inverse temperature $a$. 
\end{remark}

\subsection{Theoretical tools}\label{sec:theoretical-tools}

We write 
$\mathcal M(\overline\Omega)=C(\overline\Omega)^*$ for the finite signed Radon measures and use
$L^\infty_{\rm w*}(0,T;\mathcal M(\overline\Omega))$ for weak-star
measurable, essentially bounded measure-valued maps. Since
$C(\overline\Omega)$, and hence
$L^1(0,T;C(\overline\Omega))$, is separable, this space is identified
isometrically with $\bigl[L^1(0,T;C(\overline\Omega))\bigr]^*$.
Accordingly, $\mu_j\stackrel{*}{\rightharpoonup}\mu$ in this space means $$
 \int_0^T\pair{\mu_j(t)}{\Phi(t)}\dtm
 \longrightarrow
 \int_0^T\pair{\mu(t)}{\Phi(t)}\dtm
 \qquad
 \text{for every }\Phi\in L^1(0,T;C(\overline\Omega)).
$$ 
We refer to
\cite[Preliminary Material, Thm.~3]{FeireislLukacovaMizerovaShe2021}
for this duality and to
\cite[Sects.~5.1.1 and~5.1.3.2]{FeireislLukacovaMizerovaShe2021}
for concentration defects and their time-dependent energy formulation. Here $\langle a,b \rangle$ denotes the dual pairing in space while $(a,b)$ will denote the $L^2$ inner product in space.

\begin{lemma}[Brouwer degree continuation principle]
\label{lem:topological-degree}
Let $V\subset\mathbb R^m$ be open, let $W\Subset V$ be bounded and open,
and let
$$
 \mathbf G:\overline W\times[0,1]\to\mathbb R^m
$$
be continuous. Suppose that $\mathbf b\notin
\mathbf G(\partial W,\alpha)$ for every $\alpha\in[0,1]$ and that
$$
 \deg(\mathbf G(\cdot,0),W,\mathbf b)\ne0.
$$
Then $\mathbf G(v,1)=\mathbf b$ has at least one solution in $W$.
\end{lemma}

\begin{proof}
By homotopy invariance,
$$
 \deg(\mathbf G(\cdot,1),W,\mathbf b)
 =\deg(\mathbf G(\cdot,0),W,\mathbf b)\ne0,
$$
and the existence property of the Brouwer degree gives the conclusion.
We refer to \cite[Chap.~1]{deimling1985nonlinear} and
\cite[Thm.~2.5]{gallouet2008unconditionally}.
\end{proof}

For a grid function $v_h^n$, define on $(t^n,t^{n+1}]$
$$
 v_{h,\tau}^+=v_h^{n+1},
 \qquad
 v_{h,\tau}^-=v_h^n,
 \qquad
 \widetilde v_{h,\tau}(t)
 =\frac{t^{n+1}-t}{\tau}v_h^n
  +\frac{t-t^n}{\tau}v_h^{n+1}.
$$
We generally omit the superscript $+$ for the forward interpolant.

\begin{lemma}[Discrete Aubin--Lions lemma]
\label{lem:discrete-aubin-lions}
Let $\{u_h^n\}_{n=0}^N\subset V_h$ be a family of discrete functions,
with $N\tau=T$ and $\tau\to0$, satisfying
\begin{equation}
 \tau\sum_{n=0}^{N}\norm{u_h^n}_{H^1(\Omega)}^2
 +\tau\sum_{n=0}^{N-1}
 \norm{\deltat u_h^{n+1}}_{-1,h}^2\le C,
 \label{eq:discrete-AL-assumptions}
\end{equation}
with a constant $C$ independent of the discretisation parameters.
Then the forward, backward, and affine interpolants are relatively compact
in $L^2(\QT)$. More precisely, every sequence of discretisations admits a
subsequence along which all three interpolants converge strongly in
$L^2(\QT)$ to the same limit.
\end{lemma}

\begin{proof}
We apply the discrete Aubin--Simon lemma of
\cite[Theorem~3.4 and Remark~6]{gallouet2012compactness}.
In the notation of that result, we take
$$
 B=L^2(\Omega),\qquad
 B^{(h)}=V_h,\qquad
 \norm{v_h}_{X^{(h)}}=\norm{v_h}_{H^1(\Omega)},\qquad
 \norm{v_h}_{Y^{(h)}}=\norm{v_h}_{-1,h},
$$
and $q=2$. By the definition of the discrete negative norm,
$$
 \norm{f_h}_{-1,h}
 =
 \sup_{\substack{0\neq v_h\in V_h}}
 \frac{|(f_h,v_h)|}{\norm{v_h}_{H^1(\Omega)}},
$$
the norms $\norm{\cdot}_{H^1(\Omega)}$ and
$\norm{\cdot}_{-1,h}$ are dual with respect to the
$L^2(\Omega)$ inner product. Hence assumption \emph{(H2)} may be replaced
by \emph{(H'2)} in the sense of
\cite[Remark~6]{gallouet2012compactness}.
Finally, \eqref{eq:discrete-AL-assumptions} is precisely the uniform bound
required in \cite[Theorem~3.4]{gallouet2012compactness}, up to the harmless
additional endpoint terms present in our estimate. Thus the forward
piecewise constant interpolants $u_{h,\tau}^+$ are relatively compact in
$L^2(\QT)$.

It remains to identify the other temporal interpolants. On every interval
$(t^n,t^{n+1}]$,
$$
 u_{h,\tau}^+-u_{h,\tau}^-
 =u_h^{n+1}-u_h^n
 =\tau\deltat u_h^{n+1}.
$$
Consequently,
$$
 \begin{aligned}
 \norm{u_{h,\tau}^+-u_{h,\tau}^-}_{L^2(0,T;-1,h)}^2
 &=
 \tau\sum_{n=0}^{N-1}
 \norm{u_h^{n+1}-u_h^n}_{-1,h}^2 =
 \tau^2
 \left(
 \tau\sum_{n=0}^{N-1}
 \norm{\deltat u_h^{n+1}}_{-1,h}^2
 \right)
 \le C\tau^2 .
 \end{aligned}
$$
On the other hand, the first term in
\eqref{eq:discrete-AL-assumptions} yields
$$
 \norm{u_{h,\tau}^+-u_{h,\tau}^-}_{L^2(0,T;H^1(\Omega))}
 \le C.
$$
Using \eqref{eq:discrete-interpolation} pointwise in time and then
Cauchy--Schwarz, we obtain
$$
 \begin{aligned}
 \norm{u_{h,\tau}^+-u_{h,\tau}^-}_{L^2(\QT)}^2
 &\le
 \norm{u_{h,\tau}^+-u_{h,\tau}^-}_{L^2(0,T;H^1(\Omega))}
 \norm{u_{h,\tau}^+-u_{h,\tau}^-}_{L^2(0,T;-1,h)}\le C\tau .
 \end{aligned}
$$
Thus, it yields
$u_{h,\tau}^+-u_{h,\tau}^-\to0$
strongly in $L^2(\QT)$.

Finally, on $(t^n,t^{n+1}]$ the affine interpolant satisfies
$$
 \widetilde u_{h,\tau}
 =
 (1-\theta)u_{h,\tau}^-+\theta u_{h,\tau}^+,
 \qquad
 \theta=\frac{t-t^n}{\tau}\in[0,1].
$$
Hence, it holds
$
 |\widetilde u_{h,\tau}-u_{h,\tau}^\pm|
 \le |u_{h,\tau}^+-u_{h,\tau}^-|
$
pointwise, and therefore
$$
 \norm{\widetilde u_{h,\tau}-u_{h,\tau}^\pm}_{L^2(\QT)}
 \longrightarrow0.
$$
The relative compactness of $u_{h,\tau}^+$ therefore transfers to the
backward and affine interpolants, and all three converge along the same
subsequence to the same strong $L^2(\QT)$ limit.
\end{proof}

\begin{lemma}[Logarithmic strong monotonicity]
\label{lem:log-monotonicity}
For all $a,b>0$,
\begin{equation}
 -\bigl(\beta(a)-\beta(b)\bigr)(a-b)
 \ge \eps_1(a-b)(\ln a-\ln b)
 \ge4\eps_1(\sqrt a-\sqrt b)^2.
 \label{eq:log-monotonicity}
\end{equation}
\end{lemma}

\begin{proof}
Since
$-\beta'(s)=s^{-2}+\eps_1s^{-1}+q\eps_qs^{-q-1}\ge\eps_1s^{-1}$,
integration between $a$ and $b$ gives the first inequality.  For the second inequality, set $x=\sqrt{a/b}$.  It is equivalent to
$$
 (x-1)\bigl((x+1)\ln x-2(x-1)\bigr)\ge0.
$$
The function $g(x)=(x+1)\ln x-2(x-1)$ is increasing because
$g'(x)=\ln x-1+x^{-1}\ge0$, and $g(1)=0$.  Thus $g(x)$ has the same sign
as $x-1$, which proves the claim.
\end{proof}

\begin{lemma}[Poincar\'e inequality with control on a subset]
\label{lem:poincare-subset}
For every $m_0>0$ there is $C=C(\Omega,m_0)$ such that, whenever
$E\subset\Omega$ is measurable, $|E|\ge m_0$, and $v\in H^1(\Omega)$,
\begin{equation}
 \norm{v}_{L^2(\Omega)}
 \le C\bigl(\norm{\nabla v}_{L^2(\Omega)}
             +\norm{v}_{L^2(E)}\bigr).
 \label{eq:poincare-subset}
\end{equation}
\end{lemma}

\begin{proof}
Let $\bar v=|\Omega|^{-1}\int_\Omega v\dx$.  Then
$$
 |E|^{1/2}|\bar v|
 \le\norm{v}_{L^2(E)}+\norm{v-\bar v}_{L^2(E)}
 \le\norm{v}_{L^2(E)}+C\norm{\nabla v}_{L^2(\Omega)}.
$$
Combine this with the mean-zero Poincar\'e inequality.
\end{proof}

\begin{lemma}[Common time-slice extraction]
\label{lem:time-slice-extraction}
If $f_j\to f$ strongly in $L^p(0,T;X)$ for a separable Banach space $X$ and
$1\le p<\infty$, then a subsequence converges in $X$ for almost every time.
If the same sequence is bounded in $L^\infty(0,T;Y)$, where
$Y\hookrightarrow X$ continuously and $Y$ is reflexive, the subsequence may
be chosen so that $f_j(t)\rightharpoonup f(t)$ weakly in $Y$ for almost
every time.  One subsequence works for finitely many sequences.
\end{lemma}

\begin{proof}
Choose a subsequence with
$\sum_j\norm{f_j-f}_{L^p(0,T;X)}^p<\infty$.  Fubini's theorem gives
pointwise summability, hence convergence in $X$, for almost every time.  At
times where the $Y$ bound holds, every weak $Y$ accumulation point has the
same $X$ limit, which proves weak convergence of the full subsequence.
\end{proof}

\subsection{Generalised dissipative weak solutions}
\label{sec:weak-formulation}

We now specify the limiting solution concept used in the existence theorem.
The possible concentration of the critical singular internal energy is
recorded by a non-negative measure-valued defect.

\begin{definition}[Generalised dissipative weak solution]
\label{def:weak-solution}
Given initial data $(\rho_0,\theta_0,\eta_0)$, a sextuple
$(\rho,\mu_\rho,\theta,\eta,\mu_\eta,\lambda)$ is a generalised
dissipative weak solution if
\begin{align*}
 &\rho,\eta\in L^\infty(0,T;H^1(\Omega)),
 \qquad \mu_\rho\in L^2(0,T;H^1(\Omega)),\\
 &\theta\in L^2(0,T;H^1(\Omega))
       \cap L^\infty(0,T;L^1(\Omega)),
 \qquad \theta^{-q}\in L^\infty(0,T;L^1(\Omega)),\\
 &\mu_\eta\in L^2(\QT),
 \qquad \theta>0\text{ a.e.},
 \qquad
 \lambda\in L^\infty_{\rm w*}(0,T;\mathcal M_+(\overline\Omega)).
\end{align*}
Set, with all mobility coefficients evaluated at $(z,\theta)$, the Onsager fluxes
\begin{align*}
 \mathcal J_\rho
 &=\LL_{11}\nabla\mu_\rho-\LL_{12}\nabla\theta
   +\LL_{13}\mu_\eta,\\
 \mathcal J_e
 &=\LL_{12}^{\top}\nabla\mu_\rho-\LL_{22}\nabla\theta
   +\LL_{23}\mu_\eta,\\
 \mathcal R_\eta
 &=\LL_{13}\cdot\nabla\mu_\rho-\LL_{23}\cdot\nabla\theta
   +\LL_{33}\mu_\eta.
\end{align*}
We denote the total internal-energy measure via
\begin{equation}
 \mathfrak E_t=e(z(t),\theta(t))\,\mathrm dx+\lambda_t.
 \label{eq:total-energy-measure}
\end{equation}
For $v,\xi,w\in C^1([0,T];C^\infty_{\rm per}(\Omega))$ with terminal
value zero, and for
$\varphi,\chi\in C_c^\infty((0,T);C^\infty_{\rm per}(\Omega))$,
\begin{align}
 -\int_0^T(\rho,\partial_t v)\dtm-(\rho_0,v(0))
 +\int_0^T(\mathcal J_\rho,\nabla v)\dtm&=0,
 \label{eq:weak-rho}
\end{align}
\begin{align}
 &-\int_0^T(e(z,\theta),\partial_t\xi)\dtm
 -\int_0^T\pair{\lambda_t}{\partial_t\xi(t)}\dtm
  -(e(z_0,\theta_0),\xi(0))
 +\int_0^T(\mathcal J_e,\nabla\xi)\dtm=0,
 \label{eq:weak-energy}
\end{align}
\begin{align}
 -\int_0^T(\eta,\partial_t w)\dtm-(\eta_0,w(0))
 +\int_0^T(\mathcal R_\eta,w)\dtm&=0,
 \label{eq:weak-eta}
\end{align}
\begin{align}
 \int_0^T(\mu_\rho,\varphi)\dtm
 &=\gamma_\rho\int_0^T(\nabla\rho,\nabla\varphi)\dtm
   +\int_0^T(\partial_\rho\psi(z,\theta),\varphi)\dtm,
 \label{eq:weak-murho}\\
 \int_0^T(\mu_\eta,\chi)\dtm
 &=\gamma_\eta\int_0^T(\nabla\eta,\nabla\chi)\dtm
   +\int_0^T(\partial_\eta\psi(z,\theta),\chi)\dtm.
 \label{eq:weak-mueta}
\end{align}
For almost every $t\in(0,T)$,
\begin{equation}
 \Stot(z(t),\theta(t))
 \ge\Stot(z_0,\theta_0)
 +\int_0^t\int_\Omega Z^\top\LL(z,\theta)Z\dx\,\mathrm ds.
 \label{eq:weak-entropy}
\end{equation}
\end{definition}

\begin{remark}
 As usual, the requirements on the test functions can be weakened drastically to suitable Sobolev regularity, see Lemma \ref{lem:time-representatives}.   
\end{remark}

\section{Existence and convergence of generalised dissipative weak solutions}
\label{sec:existence}

For the existence theory, we impose
\begin{equation}
 \rho_0,\eta_0,\theta_0\in H^1(\Omega),
 \qquad \theta_0>0\ \text{a.e.},
 \qquad \theta_0^{-q}\in L^1(\Omega).
 \label{eq:initial-data}
\end{equation}
Since $q>1$, these assumptions imply integrability of
$\theta_0^{-1}$, $\theta_0^{1-q}$, and $(-\ln\theta_0)^+$.

\begin{theorem}[Existence]
\label{thm:convergence}
Let $\Omega=\mathbb T^d$, $1\le d\le3$, and let Assumption~\ref{ass:Onsager} and
Assumption~\ref{ass:potential} hold with $q\ge d$ and given initial data that satisfies
\eqref{eq:initial-data}. Then there exists a generalised dissipative weak solution
$(\rho,\mu_\rho,\theta,\eta,\mu_\eta,\lambda)$ in the sense of
Definition~\ref{def:weak-solution}. Moreover, with the continuous representatives constructed below, the solution conserves mass and internal energy:
\begin{align}
 \int_\Omega\rho(t)\dx&=\int_\Omega\rho_0\dx
 & \mathfrak E_t(\overline\Omega)
 &=\int_\Omega e(z_0,\theta_0)\dx &\text{for every }t\in[0,T],
\end{align}
\end{theorem}

\paragraph{Proof strategy.}
The proof is constructive: 
\begin{enumerate}
\item Subsections~\ref{sec:scheme}--\ref{sec:existence-fixed-h}
introduce the consistently integrated finite-element scheme, derive its exact
mass and internal-energy conservation and discrete entropy inequality, and use the coercive availability estimate together with the dimension-adapted nodal
barrier in a Brouwer degree argument. This yields a positive global discrete
trajectory. 

\item Subsections~\ref{sec:phase-compactness}--\ref{sec:defect} establish
strong compactness of the phase variables and of the inverse temperature.  The
lower-order thermal nonlinearities are then strongly compact, whereas the
critical inverse power may retain a non-negative concentration defect.  

\item Theorem~\ref{thm:conv_to_limit}
collects the resulting subsequential convergences, and
Subsection~\ref{sec:limit} passes to the limit in the balance laws and entropy
inequality.

\item The additional hypotheses under which the concentration defect
vanishes are given in Proposition~\ref{prop:no-defect}.
\end{enumerate}

\subsection{Consistent fully discrete finite-element scheme}
\label{sec:scheme}

Let $\{\mathcal T_h\}_{h>0}$ be a shape-regular, quasi-uniform family of
periodic simplicial meshes of $\Omega$, and let $\Vh$ be the continuous
piecewise affine periodic finite-element space.  Its nodal interpolation
operator is denoted by $\Ih$.  All pairings below are the ordinary,
exactly integrated $L^2$ pairings.  When one factor is only in $L^1$,
$(f,v_h)$ means the exact integral $\int_\Omega f v_h\dx$, which is
well-defined because $v_h$ is bounded for fixed $h$.

For a function $f$ that acts on $\Vh$, define
\begin{equation}
 \norm{f}_{-1,h}
 :=\sup_{0\ne v_h\in\Vh}
 \frac{|(f,v_h)|}{\norm{v_h}_{H^1(\Omega)}}.
 \label{eq:discrete-dual}
\end{equation}
For $v_h\in\Vh$ one has the elementary interpolation estimate
\begin{equation}
 \norm{v_h}_{L^2}^2
 \le \norm{v_h}_{H^1}\norm{v_h}_{-1,h}.
 \label{eq:discrete-interpolation}
\end{equation}
For every smooth periodic test function $\phi$, the nodal interpolant satisfies
\begin{equation}
 \norm{\Ih\phi-\phi}_{W^{1,\infty}(\Omega)}\to0.
 \label{eq:smooth-interpolation}
\end{equation}

Let $t^n=n\tau$, $n=0,\dots,N$, $N\tau=T$, and set
$$
 \deltat v_h^{n+1}:=\frac{v_h^{n+1}-v_h^n}{\tau}.
$$
We use the exact nonlinear functions
\begin{align}
 e_h^n&:=e(z_h^n,\theta_h^n),
 \label{eq:discrete-energy-density}\\
 G_h^{n+1}
 &:=D_z\psi_\vex(z_h^{n+1},\theta_h^{n+1})
   +D_z\psi_\cav(z_h^n,\theta_h^{n+1}).
 \label{eq:discrete-split}
\end{align}
The mobility is evaluated at the old state,
$\LL_h^n=\LL(z_h^n,\theta_h^n)$.

\begin{definition}[Fully discrete problem]
\label{def:fully-discrete}
Given $(\rho_h^n,\theta_h^n,\eta_h^n)\in\Vh^3$ with
$\theta_h^n>0$ on $\overline\Omega$, find
$$
 (\rho_h^{n+1},\mu_{\rho,h}^{n+1},\theta_h^{n+1},
  \eta_h^{n+1},\mu_{\eta,h}^{n+1})\in\Vh^5,
 \qquad \theta_h^{n+1}>0\ \text{on }\overline\Omega,
$$
such that, for all $(v_h,\zeta_h,\xi_h,w_h,\chi_h)\in\Vh^5$,
\begin{subequations}
\label{eq:fully-discrete}
\begin{align}
 (\deltat\rho_h^{n+1},v_h)
 &+\bigl(\LL_{11,h}^n\nabla\mu_{\rho,h}^{n+1}
          -\LL_{12,h}^n\nabla\theta_h^{n+1}
          +\LL_{13,h}^n\mu_{\eta,h}^{n+1},\nabla v_h\bigr)=0,
 \label{eq:fd-rho}\\
 (\mu_{\rho,h}^{n+1},\zeta_h)
 &=\gamma_\rho(\nabla\rho_h^{n+1},\nabla\zeta_h)
   +(G_{\rho,h}^{n+1},\zeta_h),
 \label{eq:fd-murho}\\
 (\deltat e_h^{n+1},\xi_h)
 &+\bigl((\LL_{12,h}^n)^\top\nabla\mu_{\rho,h}^{n+1}
          -\LL_{22,h}^n\nabla\theta_h^{n+1}
          +\LL_{23,h}^n\mu_{\eta,h}^{n+1},\nabla\xi_h\bigr)=0,
 \label{eq:fd-energy}\\
 (\deltat\eta_h^{n+1},w_h)
 &+\bigl(\LL_{13,h}^n\cdot\nabla\mu_{\rho,h}^{n+1}
          -\LL_{23,h}^n\cdot\nabla\theta_h^{n+1}
          +\LL_{33,h}^n\mu_{\eta,h}^{n+1},w_h\bigr)=0,
 \label{eq:fd-eta}\\
 (\mu_{\eta,h}^{n+1},\chi_h)
 &=\gamma_\eta(\nabla\eta_h^{n+1},\nabla\chi_h)
   +(G_{\eta,h}^{n+1},\chi_h).
 \label{eq:fd-mueta}
\end{align}
\end{subequations}
\end{definition}

\begin{remark}[Exact integration]
\label{rem:exact-integration}
All nonlinear terms in Definition~\ref{def:fully-discrete} are understood as
exact integrals.  This convention is part of the analysed method: replacing
the non-polynomial thermal terms by a fixed quadrature rule generally
introduces residuals into the discrete energy and entropy identities.  A
quadrature-based implementation therefore requires a separate consistency
and stability analysis.
\end{remark}

Define
\begin{align}
 \Energy_h^n&:=\int_\Omega e_h^n\dx,
 \label{eq:Eh}\\
 \Stot_h^n
 &:=\int_\Omega\Sbulk(z_h^n,\theta_h^n)\dx
 -\frac{\gamma_\rho}{2}\norm{\nabla\rho_h^n}_{L^2}^2
 -\frac{\gamma_\eta}{2}\norm{\nabla\eta_h^n}_{L^2}^2,
 \label{eq:Sh}\\
 Z_h^{n+1}&=(\nabla\mu_{\rho,h}^{n+1},
             -\nabla\theta_h^{n+1},
             \mu_{\eta,h}^{n+1}),
 \label{eq:Zh}\\
 \Diss_h^{n+1}&:=\int_\Omega
 (Z_h^{n+1})^\top\LL_h^n Z_h^{n+1}\dx.
 \label{eq:Dh}
\end{align}

\begin{theorem}[Discrete thermodynamic structure]
\label{thm:discrete-structure}
Every solution of Definition~\ref{def:fully-discrete} satisfies
\begin{align}
 \int_\Omega\rho_h^{n+1}\dx
 &=\int_\Omega\rho_h^n\dx,
 \label{eq:mass-conservation}\\
 \Energy_h^{n+1}&=\Energy_h^n,
 \label{eq:energy-conservation}\\
 \Stot_h^{n+1}-\Stot_h^n
 &\ge\tau\Diss_h^{n+1}
 +\frac{\gamma_\rho}{2}\norm{\nabla(\rho_h^{n+1}-\rho_h^n)}_{L^2}^2
 +\frac{\gamma_\eta}{2}\norm{\nabla(\eta_h^{n+1}-\eta_h^n)}_{L^2}^2.
 \label{eq:entropy-inequality}
\end{align}
\end{theorem}

\begin{proof}
The proof is already given in \cite{BrunkHabrichOyedejiYangXu}, but for completeness we present the main arguments here.
Testing \eqref{eq:fd-rho} and \eqref{eq:fd-energy} with $1$ gives
\eqref{eq:mass-conservation}--\eqref{eq:energy-conservation}.
Fix $x\in\Omega$ and abbreviate
$z^+=z_h^{n+1}(x)$, $z^-=z_h^n(x)$,
$\theta^+=\theta_h^{n+1}(x)$, and $\theta^-=\theta_h^n(x)$.
Concavity of $\psi(z^-,\cdot)$ gives
$$
 \psi(z^-,\theta^-)-\psi(z^-,\theta^+)
 +(\theta^+-\theta^-)e(z^-,\theta^-)\ge0.
$$
At the common new temperature, convexity of $\psi_\vex$ and concavity of
$\psi_\cav$ give
$$
 \psi(z^-,\theta^+)-\psi(z^+,\theta^+)
 \ge-G_h^{n+1}(x)\cdot(z^+-z^-).
$$
Using $\Sbulk=\theta e-\psi$ and adding the two inequalities yields the
pointwise estimate
\begin{align}
 \Sbulk(z^+,\theta^+)-\Sbulk(z^-,\theta^-)
 \ge{}&\theta^+\bigl(e(z^+,\theta^+)-e(z^-,\theta^-)\bigr)
 -G_h^{n+1}\cdot(z^+-z^-).
 \label{eq:pointwise-entropy-increment}
\end{align}
Integrate \eqref{eq:pointwise-entropy-increment}.  Testing
\eqref{eq:fd-murho} and \eqref{eq:fd-mueta} with the two phase increments
and using
$|A|^2-|B|^2=2A\cdot(A-B)-|A-B|^2$ gives
\begin{align*}
 \Stot_h^{n+1}-\Stot_h^n
 \ge{}&(\theta_h^{n+1},e_h^{n+1}-e_h^n)
 -(\mu_{\rho,h}^{n+1},\rho_h^{n+1}-\rho_h^n)-(\mu_{\eta,h}^{n+1},\eta_h^{n+1}-\eta_h^n)\\
 &+\frac{\gamma_\rho}{2}\norm{\nabla(\rho_h^{n+1}-\rho_h^n)}_{L^2}^2
 +\frac{\gamma_\eta}{2}\norm{\nabla(\eta_h^{n+1}-\eta_h^n)}_{L^2}^2.
\end{align*}
Finally, test \eqref{eq:fd-rho}, \eqref{eq:fd-energy}, and
\eqref{eq:fd-eta} with $-\mu_{\rho,h}^{n+1}$,
$\theta_h^{n+1}$, and $-\mu_{\eta,h}^{n+1}$, respectively.  Symmetry of
$\LL_h^n$ gives
$$
 (\theta_h^{n+1},e_h^{n+1}-e_h^n)
 -(\mu_{\rho,h}^{n+1},\rho_h^{n+1}-\rho_h^n)
 -(\mu_{\eta,h}^{n+1},\eta_h^{n+1}-\eta_h^n)
 =\tau\Diss_h^{n+1}.
$$
This proves \eqref{eq:entropy-inequality}. 
\end{proof}

\subsection{Discrete availability, full norms, and time-derivative estimates}
\label{sec:estimates}

Set
\begin{equation}
 \Avail_{a,h}^n:=a\Energy_h^n-\Stot_h^n.
 \label{eq:discrete-availability}
\end{equation}
By Theorem~\ref{thm:discrete-structure},
\begin{equation}
 \Avail_{a,h}^{n+1}+\tau\Diss_h^{n+1}
 \le\Avail_{a,h}^n.
 \label{eq:availability-dissipation}
\end{equation}

\begin{theorem}[Uniform discrete estimates]
\label{thm:uniform-estimates}
Assume that the initial approximations have uniformly bounded availability.
Then a constant $C$, independent of $h$, $\tau$, and $n$, satisfies
\begin{align}
 &\sup_{0\le n\le N}
 \Bigl[
  \norm{\rho_h^n}_{H^1}^2+\norm{\eta_h^n}_{H^1}^2
  +\norm{\rho_h^n}_{L^4}^4+\norm{\eta_h^n}_{L^4}^4
  +\norm{\theta_h^n}_{L^1}
  +\int_\Omega(\theta_h^n)^{-q}\dx
 \Bigr]
 \nonumber\\
 &\quad
 +\tau\sum_{n=0}^{N-1}
 \Bigl[
  \norm{\mu_{\rho,h}^{n+1}}_{H^1}^2
  +\norm{\theta_h^{n+1}}_{H^1}^2
  +\norm{\mu_{\eta,h}^{n+1}}_{L^2}^2
 \Bigr]
 \le C.
 \label{eq:uniform-estimates}
\end{align}
In addition,
\begin{equation}
 \sup_n\left(
  \norm{e_h^n}_{L^1(\Omega)}
  +\norm{\Sbulk(z_h^n,\theta_h^n)}_{L^1(\Omega)}
 \right)\le C.
 \label{eq:energy-entropy-L1}
\end{equation}
\end{theorem}

\begin{proof}
The availability has the exact representation
\begin{align*}
 \Avail_{a,h}^n
 ={}&\int_\Omega
 [a e(z_h^n,\theta_h^n)-\Sbulk(z_h^n,\theta_h^n)]\dx+\frac{\gamma_\rho}{2}\norm{\nabla\rho_h^n}_{L^2}^2
  +\frac{\gamma_\eta}{2}\norm{\nabla\eta_h^n}_{L^2}^2.
\end{align*}
Hence \eqref{eq:availability-dissipation},
Lemma~\ref{lem:availability-coercivity}, and the ellipticity of $\LL$ give
all time-uniform state bounds and
$$
 \tau\sum_n\left(
 \norm{\nabla\mu_{\rho,h}^{n+1}}_{L^2}^2
 +\norm{\nabla\theta_h^{n+1}}_{L^2}^2
 +\norm{\mu_{\eta,h}^{n+1}}_{L^2}^2\right)\le C.
$$
The Poincar\'e inequality with $L^1$ control, that is
 $$\norm{v}_{L^2(\Omega)}
 \le C_\Omega\bigl(\norm{\nabla v}_{L^2(\Omega)}
                    +\norm{v}_{L^1(\Omega)}\bigr)\qquad \forall v \in H^1(\Omega),$$
then yields
$$
 \norm{\theta_h^{n+1}}_{L^2}^2
 \le C\left(\norm{\nabla\theta_h^{n+1}}_{L^2}^2+1\right),
$$
and summation in time gives the full $L^2(0,T;H^1)$ estimate.

To control the mean of $\mu_{\rho,h}^{n+1}$, test
\eqref{eq:fd-murho} with $1$.  By \eqref{eq:split-growth}, H\"older's
inequality, and $H^1(\Omega)\hookrightarrow L^6(\Omega)$,
\begin{align*}
 \left|\frac1{|\Omega|}\int_\Omega\mu_{\rho,h}^{n+1}\dx\right|
 &\le C\int_\Omega(1+\theta_h^{n+1})
  (1+|z_h^{n+1}|^3+|z_h^n|^3)\dx\le C\bigl(1+\norm{\theta_h^{n+1}}_{L^2}\bigr).
\end{align*}
The ordinary Poincar\'e inequality and the gradient dissipation therefore
yield the full $H^1$ estimate for $\mu_{\rho,h}$.  Finally,
\eqref{eq:energy-entropy-L1} follows directly from
\eqref{eq:e-s-growth}.
\end{proof}

\begin{corollary}[Uniform temperature tails]
\label{cor:tails}
For every $\kappa\in(0,1)$ and $R>1$,
\begin{equation}
 \sup_n|\{\theta_h^n<\kappa\}|\le C\kappa^q,
 \qquad
 \sup_n|\{\theta_h^n>R\}|\le \frac{C}{R}.
 \label{eq:temperature-tails}
\end{equation}
No mesh-uniform pointwise lower or upper bound is asserted.
\end{corollary}

\begin{proposition}[Discrete time-derivative estimates]
\label{prop:discrete-time-derivatives}
Under the assumptions of Theorem~\ref{thm:uniform-estimates},
\begin{align}
 \tau\sum_{n=0}^{N-1}
 \norm{\deltat\rho_h^{n+1}}_{-1,h}^2&\le C,
 \label{eq:dt-rho}\\
 \tau\sum_{n=0}^{N-1}
 \norm{\deltat e_h^{n+1}}_{-1,h}^2&\le C,
 \label{eq:dt-e}\\
 \tau\sum_{n=0}^{N-1}
 \norm{\deltat\eta_h^{n+1}}_{L^2}^2&\le C.
 \label{eq:dt-eta}
\end{align}
\end{proposition}

\begin{proof}
The first two estimates follow from \eqref{eq:fd-rho} and
\eqref{eq:fd-energy}, the boundedness of the mobility, and the definition
of the discrete dual norm.  For the non-conserved variable take
$w_h=\deltat\eta_h^{n+1}$ in \eqref{eq:fd-eta}; then
$$
 \norm{\deltat\eta_h^{n+1}}_{L^2}
 \le C\Bigl(
 \norm{\nabla\mu_{\rho,h}^{n+1}}_{L^2}
 +\norm{\nabla\theta_h^{n+1}}_{L^2}
 +\norm{\mu_{\eta,h}^{n+1}}_{L^2}\Bigr).
$$
The dissipation estimate completes the proof.
\end{proof}

\subsection{Initial approximation, positivity, and fixed-mesh solvability}
\label{sec:existence-fixed-h}

\begin{lemma}[Positive, thermodynamically consistent initial data]
\label{lem:initial-approximation}
Let \eqref{eq:initial-data} hold.  There are triples
$(\rho_{0,h},\theta_{0,h},\eta_{0,h})\in\Vh^3$ with
$\theta_{0,h}>0$ on $\overline\Omega$ such that
\begin{equation}
 \rho_{0,h}\to\rho_0,\qquad
 \eta_{0,h}\to\eta_0,\qquad
 \theta_{0,h}\to\theta_0
 \quad\text{strongly in }H^1(\Omega),
 \label{eq:initial-H1-convergence}
\end{equation}
and
\begin{align}
 \int_\Omega e(z_{0,h},\theta_{0,h})\dx
 &\longrightarrow\int_\Omega e(z_0,\theta_0)\dx,
 \label{eq:initial-energy-convergence}\\
 \int_\Omega\Sbulk(z_{0,h},\theta_{0,h})\dx
 &\longrightarrow\int_\Omega\Sbulk(z_0,\theta_0)\dx.
 \label{eq:initial-entropy-convergence}
\end{align}
The conserved phase may additionally be chosen so that
$\int_\Omega\rho_{0,h}\dx=\int_\Omega\rho_0\dx$.
\end{lemma}

\begin{proof}
Set $T_m(s)=\min\{m,\max\{m^{-1},s\}\}$.  The Sobolev chain rule gives
$$
 \nabla T_m(\theta_0)
 =\mathbf 1_{\{m^{-1}<\theta_0<m\}}\nabla\theta_0
 \quad\text{a.e.}
$$
Since $\theta_0>0$ almost everywhere, dominated convergence, together with
the usual upper-truncation estimate in $L^2$, yields
$T_m(\theta_0)\to\theta_0$ strongly in $H^1(\Omega)$.  Moreover,
$T_m(\theta_0)^{-q}\to\theta_0^{-q}$ in $L^1$: on
$\{\theta_0<m^{-1}\}$ the difference is bounded by $\theta_0^{-q}$,
whereas on $\{\theta_0>m\}$ it is bounded by $m^{-q}$.

For fixed $m$, periodic convolution with a non-negative mollifier preserves
the bounds $m^{-1}\le T_m(\theta_0)\le m$ and converges in $H^1$.  On this
compact interval all reciprocal powers and the logarithm are Lipschitz.
Mollifying the phase variables as well and taking a diagonal sequence gives
smooth periodic triples $(z_{0,j},\theta_{0,j})$ such that
\begin{equation}
 z_{0,j}\to z_0\ \text{in }H^1,
 \qquad
 \theta_{0,j}\to\theta_0\ \text{in }H^1,
 \qquad
 \theta_{0,j}^{-q}\to\theta_0^{-q}\ \text{in }L^1,
 \label{eq:smooth-initial-diagonal}
\end{equation}
with $\min_{\overline\Omega}\theta_{0,j}>0$ for every $j$.
After extraction, all state variables converge almost everywhere.

The singular convergence in \eqref{eq:smooth-initial-diagonal} controls the
remaining thermal terms quantitatively.  The elementary inequality
$|r^{1/q}-s^{1/q}|^q\le |r-s|$ for $r,s\ge0$ gives
$$
 \theta_{0,j}^{-1}\to\theta_0^{-1}
 \quad\text{strongly in }L^q(\Omega).
$$
It follows that
$\theta_{0,j}^{1-q}\to\theta_0^{1-q}$ strongly in $L^1$.  Furthermore,
for every $p>1$ there is $C_{p,q}$ such that
$[(-\ln s)^+]^p\le C_{p,q}(1+s^{-q})$ for $s>0$; hence the negative
logarithms are uniformly integrable and converge strongly in $L^1$ by
Vitali's theorem.  The map $s\mapsto(\ln s)^+$ is one-Lipschitz on
$(0,\infty)$, so the positive logarithms converge strongly in $L^2$.
Finally, $H^1(\Omega)\hookrightarrow L^4(\Omega)$ for $d\le3$, and thus
the phase polynomials converge strongly in $L^1$.  The explicit formulas
\eqref{eq:energy-density}--\eqref{eq:entropy-density} now imply convergence
of the smooth bulk energies and entropies in $L^1$.

For each fixed smooth triple, nodal interpolation converges in $H^1$ and
uniformly.  Consequently, for all sufficiently small $h$ the interpolated
inverse temperature remains strictly positive, and all its thermodynamic
compositions converge uniformly.  A second diagonal choice of the smoothing
index as $h\downarrow0$ produces the asserted finite-element family and
\eqref{eq:initial-H1-convergence}--\eqref{eq:initial-entropy-convergence}.
The constant correction
$$
 \rho_{0,h}\mapsto\rho_{0,h}
 +\frac1{|\Omega|}\left(\int_\Omega\rho_0\dx
                 -\int_\Omega\rho_{0,h}\dx\right)
$$
enforces the exact mass constraint.  Its size tends to zero, so it does not
alter any of the established convergences.
\end{proof}

\begin{lemma}[Dimension-adapted finite-element positivity barrier]
\label{lem:fe-positivity-barrier}
Fix a simplicial mesh $\mathcal T_h$ in dimension $d$ and let $q\ge d$.
For every $C_0>0$ there is $c_h(C_0)>0$ such that every
$v_h\in\Vh$ satisfying
\begin{equation}
 v_h>0\ \text{on }\overline\Omega,
 \qquad
 \int_\Omega v_h\dx+\int_\Omega v_h^{-q}\dx\le C_0
 \label{eq:barrier-assumptions}
\end{equation}
obeys
\begin{equation}
 \min_{\overline\Omega}v_h\ge c_h(C_0).
 \label{eq:barrier-conclusion}
\end{equation}
The constant is permitted to depend on the fixed mesh.  If $q<d$, the
conclusion is false in general.
\end{lemma}

\begin{proof}
Because $v_h$ is affine on every simplex, its minimum is attained at a
vertex.  Let $a_0$ be a minimum vertex and put
$\delta=v_h(a_0)$.  Positivity and the exact affine identity
$$
 \int_\Omega v_h\dx
 =\sum_{a\in\mathcal N_h}\left(\int_\Omega\varphi_a\dx\right)v_h(a)
$$
give a fixed-mesh upper bound $v_h(a)\le M_h(C_0)$ at every vertex.
Choose a simplex $K$ adjacent to $a_0$, let $\lambda_0$ be the barycentric
coordinate of $a_0$, and set $r=1-\lambda_0$.  Then
$$
 v_h(x)\le\delta+M_h r\qquad(x\in K).
$$
The affine change of variables on a simplex gives, for every non-negative
function $f$,
$$
 \int_K f(1-\lambda_0(x))\dx
 =d|K|\int_0^1 f(r)r^{d-1}\,\mathrm dr.
$$
Consequently,
\begin{equation}
 \int_\Omega v_h^{-q}\dx
 \ge d|K|\int_0^1\frac{r^{d-1}}{(\delta+M_h r)^q}\,\mathrm dr.
 \label{eq:barrier-integral}
\end{equation}
If $q>d$, integration over $0<r<\delta/M_h$ shows that the right-hand side
is bounded below by $c_h\delta^{d-q}$.  If $q=d$, integration over
$\delta/M_h<r<1/2$ gives the lower bound
$c_h|\ln\delta|-C_h$.  In either case the right-hand side diverges as
$\delta\downarrow0$, contradicting \eqref{eq:barrier-assumptions}.  This
proves \eqref{eq:barrier-conclusion}.

For $q<d$, prescribe the nodal value $\delta$ at one vertex $a_0$ of a
fixed mesh and the value $1$ at every other vertex.  On each simplex
adjacent to $a_0$ the resulting finite-element function has the form
$v_\delta=\delta+(1-\delta)(1-\lambda_0)$, and on all remaining simplices
it equals $1$.  The formula above shows that
$\int_\Omega v_\delta^{-q}\dx$ remains bounded while
$v_\delta(a_0)=\delta\downarrow0$.  Hence the exponent condition is sharp
for this fixed-mesh argument.
\end{proof}

\begin{theorem}[One-step existence for the consistent scheme]
\label{thm:one-step-existence}
For every fixed $h>0$, $\tau>0$, and old state
$(\rho_h^n,\theta_h^n,\eta_h^n)\in\Vh^3$ with
$\theta_h^n>0$ on $\overline\Omega$, problem
\eqref{eq:fully-discrete} has at least one solution with
$\theta_h^{n+1}>0$ on $\overline\Omega$.
\end{theorem}

\begin{proof}
Choose coordinates in $\Vh^5$ and denote the resulting Euclidean space by
$X_h$.  For $\alpha\in[0,1]$, let $\mathcal F_\alpha$ be the residual of
\eqref{eq:fully-discrete} after multiplication of the three Onsager terms
in the balance equations by $\alpha$.  Since all thermal functions are
smooth on $(0,\infty)$, the map $(\alpha,U)\mapsto\mathcal F_\alpha(U)$ is
continuous, and is continuously differentiable in $U$, on
$$
 [0,1]\times\mathcal O_h,
 \qquad
 \mathcal O_h=\{U\in X_h:\theta>0\text{ on }\overline\Omega\}.
$$
Every zero of $\mathcal F_\alpha$ satisfies
\begin{equation}
 \Avail_{a,h}^{n+1}+\alpha\tau\Diss_h^{n+1}
 \le\Avail_{a,h}^n.
 \label{eq:degree-availability}
\end{equation}
The coercivity lemma gives bounds, independent of $\alpha$, on
$\int\theta_h^{n+1}$, $\int(\theta_h^{n+1})^{-q}$, and
$\|z_h^{n+1}\|_{L^4}$.  Lemma~\ref{lem:fe-positivity-barrier} and the exact
affine identity for the integral give
$$
 0<c_h\le\theta_h^{n+1}(a)\le C_h
 \qquad(a\in\mathcal N_h).
$$
Because $h$ is fixed, norm equivalence in $\Vh$ converts the $L^4$ bound
into bounds for all phase coefficients.  The two constitutive equations can
be written with the positive definite consistent mass matrix as
linear systems for the chemical-potential coefficients; their right-hand
sides are bounded by the already established state bounds.  The resulting coefficient bounds are uniform in $\alpha$.  Moreover, any
convergent sequence of zeros has a limit in $\mathcal O_h$ by the uniform
nodal lower bound, and the limit is again a zero by continuity of the
residual.  Hence the union of all zero sets,
$$
 \bigcup_{\alpha\in[0,1]}\{U\in\mathcal O_h:\mathcal F_\alpha(U)=0\},
$$
is compact and contained in a set
$\mathcal K_h\Subset\mathcal O_h$.  Choose a bounded open set
$\mathcal U_h$ such that
$\mathcal K_h\subset\mathcal U_h\Subset\mathcal O_h$.  No zero of any
$\mathcal F_\alpha$ lies on $\partial\mathcal U_h$.

At $\alpha=0$, the first and fourth equations and invertibility of the
consistent mass matrix give
$$
 \rho_h^{n+1}=\rho_h^n,
 \qquad \eta_h^{n+1}=\eta_h^n.
$$
The energy equation reduces to
$$
 (e(z_h^n,\theta_h^{n+1})-e(z_h^n,\theta_h^n),\xi_h)=0
 \qquad(\xi_h\in\Vh).
$$
Taking $\xi_h=\theta_h^{n+1}-\theta_h^n$ and using the strict decrease in
\eqref{eq:energy-monotonicity} gives
$\theta_h^{n+1}=\theta_h^n$.  The constitutive equations then determine the
two chemical potentials uniquely.  Thus $\mathcal F_0$ has exactly one zero
in $\mathcal U_h$.

It remains to check that this zero is non-degenerate.  If a vector belongs
to the kernel of $D\mathcal F_0$, the first and fourth balance rows imply
$\delta\rho=\delta\eta=0$.  The energy row then gives
$$
 \int_\Omega\partial_\theta e(z_h^n,\theta_h^n)
 |\delta\theta|^2\dx=0,
$$
and therefore $\delta\theta=0$ by
\eqref{eq:energy-monotonicity}.  The two constitutive rows, together with
positive definiteness of the mass matrix, give
$\delta\mu_\rho=\delta\mu_\eta=0$.  Hence the Jacobian is invertible and
$\deg(\mathcal F_0,\mathcal U_h,0)=\pm1$. Applying
Lemma~\ref{lem:topological-degree} with
$V=\mathcal O_h$, $W=\mathcal U_h$,
$\mathbf G(U,\alpha)=\mathcal F_\alpha(U)$, and $\mathbf b=0$
shows that $\mathcal F_1$ has a zero in $\mathcal U_h$.
\end{proof}

\begin{corollary}[Global discrete solutions]
\label{cor:global-discrete-solutions}
Starting from Lemma~\ref{lem:initial-approximation}, repeated application of
Theorem~\ref{thm:one-step-existence} produces a strictly positive fully
discrete solution on every time grid.  All estimates of
Theorem~\ref{thm:uniform-estimates} and
Proposition~\ref{prop:discrete-time-derivatives} hold uniformly in
$h$ and $\tau$.
\end{corollary}

\begin{remark}
The positivity constant in Lemma~\ref{lem:fe-positivity-barrier} depends on
the fixed mesh and is used only in the finite-dimensional degree argument.
The continuum limit uses solely the mesh-independent integral estimates.
The degree argument proves existence, but not uniqueness, of the nonlinear
algebraic solution at a time step.
\end{remark}

\subsection{Temporal interpolants and phase compactness}
\label{sec:phase-compactness}

\begin{proposition}[Strong compactness of the phase variables]
\label{prop:phase-compactness}
Along a subsequence,
\begin{align}
 \rho_{h,\tau}^{\pm}\to\rho,\quad
 \eta_{h,\tau}^{\pm}\to\eta
 &\quad\text{strongly in }L^2(\QT),
 \label{eq:phase-L2-compactness}\\
 z_{h,\tau}^{\pm}\to z
 &\quad\text{strongly in }L^p(\QT)
 \quad \text{for every }1\le p<6.
 \label{eq:phase-Lp-compactness}
\end{align}
Moreover, $\rho,\eta\in L^\infty(0,T;H^1(\Omega))$.
\end{proposition}

\begin{proof}
Apply Lemma~\ref{lem:discrete-aubin-lions} to $\rho_{h,\tau}$ using
\eqref{eq:dt-rho}.  For $\eta$, estimate \eqref{eq:dt-eta} is stronger.
Interpolation between strong $L^2$ convergence and the uniform
$L^\infty(0,T;L^6)$ bound gives \eqref{eq:phase-Lp-compactness}.
\end{proof}

\begin{lemma}[Strong compactness of the phase energy]
\label{lem:phase-energy-strong}
For both time interpolants,
\begin{equation}
 F_1(z_{h,\tau}^{\pm})\to F_1(z)
 \quad\text{strongly in }L^2(0,T;L^{6/5}(\Omega)).
 \label{eq:phase-energy-mixed-strong}
\end{equation}
Consequently,
\begin{equation}
 \lim_{s\downarrow0}\limsup_{h,\tau\to0}
 \norm{F_1(z_{h,\tau})(\cdot+s)-F_1(z_{h,\tau})}_
 {L^2(0,T-s;L^{6/5})}=0.
 \label{eq:phase-energy-translations}
\end{equation}
\end{lemma}

\begin{proof}
The phase differences converge strongly in
$L^4(0,T;L^3(\Omega))$ by interpolation between strong
$L^2(0,T;L^2)$ convergence and the uniform
$L^\infty(0,T;L^6)$ bound.  Since
$$
 |F_1(z_1)-F_1(z_2)|
 \le C(1+|z_1|^3+|z_2|^3)|z_1-z_2|
$$
and the cubic factor is bounded in $L^\infty(0,T;L^2)$, the convergence
holds even in $L^4(0,T;L^{6/5})$, hence in the space stated.  Translation
continuity of the limit and a three-term argument prove
\eqref{eq:phase-energy-translations}.
\end{proof}

\subsection{Strong compactness of the inverse temperature}
\label{sec:temperature-compactness}

Write
\begin{equation}
 e(z,\theta)=\beta(\theta)+F_1(z),
 \qquad
 \beta(\theta)=\theta^{-1}-\eps_1\ln\theta+\eps_q\theta^{-q}.
 \label{eq:separable-energy}
\end{equation}
The tuned logarithmic term supplies a global monotonicity estimate that is
particularly well suited to the consistent Galerkin scheme.

\begin{lemma}[Time translations of the internal energy]
\label{lem:energy-time-translations}
Let $e_{h,\tau}$ be the forward interpolant and
$\widetilde e_{h,\tau}$ its affine interpolant, both understood as
functionals on $\Vh$.  Then
\begin{equation}
 \norm{\partial_t\widetilde e_{h,\tau}}_{L^2(0,T;-1,h)}\le C
 \label{eq:energy-affine-bound}
\end{equation}
and, for $0<s<T$,
\begin{equation}
 \norm{e_{h,\tau}(\cdot+s)-e_{h,\tau}}_
 {L^2(0,T-s;-1,h)}\le C(s+\tau).
 \label{eq:energy-translation-bound}
\end{equation}
\end{lemma}

\begin{proof}
The first estimate is \eqref{eq:dt-e}.  The affine interpolant satisfies the
usual translation estimate $Cs$ by the fundamental theorem of calculus.
On each time cell, direct integration gives
$$
 \norm{e_{h,\tau}-\widetilde e_{h,\tau}}_{L^2(0,T;-1,h)}
 \le C\tau.
$$
The triangle inequality proves \eqref{eq:energy-translation-bound}.
\end{proof}

\begin{theorem}[Temperature compactness]
\label{thm:temperature-compactness}
Along a subsequence there is
$\theta\in L^2(0,T;H^1(\Omega))$, with $\theta>0$ almost everywhere, such
that
\begin{equation}
 \theta_{h,\tau}^{\pm}\to\theta
 \qquad\text{strongly in }L^2(\QT).
 \label{eq:theta-strong}
\end{equation}
\end{theorem}

\begin{proof}
The uniform estimates give weak convergence of the forward temperatures in
$L^2(0,T;H^1)$.  Fix $0<s<T$ and put
$a=\theta_{h,\tau}(t+s)$ and $b=\theta_{h,\tau}(t)$.  Their difference
belongs to $\Vh$.  By Lemma~\ref{lem:log-monotonicity} and
$e=\beta+F_1$,
\begin{align}
 4\eps_1\norm{\sqrt a-\sqrt b}_{L^2((0,T-s)\times\Omega)}^2
 \le{}&
 \left|\int_0^{T-s}(e_{h,\tau}(t+s)-e_{h,\tau}(t),a-b)\dtm\right|
 \nonumber\\
 &+\left|\int_0^{T-s}
 (F_1(z_{h,\tau}(t+s))-F_1(z_{h,\tau}(t)),a-b)\dtm\right|.
 \label{eq:sqrt-time-translation}
\end{align}
The first term is bounded by
$C(s+\tau)\norm{a-b}_{L^2(0,T-s;H^1)}$ using
Lemma~\ref{lem:energy-time-translations}.  The second is bounded by
$$
 \norm{F_1(z_{h,\tau})(\cdot+s)-F_1(z_{h,\tau})}_
 {L^2(0,T-s;L^{6/5})}
 \norm{a-b}_{L^2(0,T-s;L^6)}.
$$
The second factor is uniformly bounded and the first tends to zero in the
iterated limit by Lemma~\ref{lem:phase-energy-strong}.  Hence
\begin{equation}
 \lim_{s\downarrow0}\limsup_{h,\tau\to0}
 \norm{\sqrt{\theta_{h,\tau}(\cdot+s)}-
       \sqrt{\theta_{h,\tau}}}_{L^2((0,T-s)\times\Omega)}=0.
 \label{eq:sqrt-time-compactness}
\end{equation}

For spatial translations $y$ on the torus,
\begin{align*}
 &\norm{\sqrt{\theta_{h,\tau}(\cdot,\cdot+y)}-
       \sqrt{\theta_{h,\tau}}}_{L^2(\QT)}^2\le
 \norm{\theta_{h,\tau}(\cdot,\cdot+y)-\theta_{h,\tau}}_{L^1(\QT)}
 \le C|y|,
\end{align*}
where the last estimate follows from the $L^2(0,T;H^1)$ bound.  Since
$\int\theta_{h,\tau}$ is uniformly bounded, the square roots are bounded in
$L^2(\QT)$; the time-end strips have squared $L^2$ norm at most $Cs$ by the
uniform-in-time $L^1$ bound for $\theta_{h,\tau}$.  The
Fr\'echet--Kolmogorov criterion and \eqref{eq:sqrt-time-compactness} yield
$$
 \sqrt{\theta_{h,\tau}}\to r
 \quad\text{strongly in }L^2(\QT).
$$
Therefore $\theta_{h,\tau}\to r^2$ strongly in $L^1$ and in measure.  The
weak $L^2H^1$ limit identifies $\theta=r^2$.  Fatou's lemma and the
$\theta_h^{-q}$ estimate show
$$
 \int_{\QT}\theta^{-q}\dx\dtm<\infty,
$$
so $\theta>0$ almost everywhere.

It remains to upgrade to strong $L^2$.  For almost every $t$, the uniform
$L^1$ bound gives $|\{\theta_h(t)>R\}|\le C/R$.  For $R$ large enough the
set $E_h(t)=\{\theta_h(t)\le R\}$ has measure at least $|\Omega|/2$.
With $w_{h,R}=(\theta_h-R)^+$, Lemma~\ref{lem:poincare-subset} and Sobolev
embedding give
$$
 \norm{w_{h,R}(t)}_{L^{p_d}}
 \le C\norm{\nabla\theta_h(t)}_{L^2},
 \qquad
 p_d=\begin{cases}4,&d\le2,\\6,&d=3.\end{cases}
$$
Set $A_{h,R}(t)=\{\theta_h(t)>2R\}$.  On this set
$\theta_h\le2w_{h,R}$, while the $L^1$ bound gives
$|A_{h,R}(t)|\le C/R$.  H\"older's inequality therefore yields
\begin{align*}
 \int_{A_{h,R}(t)}\theta_h(t)^2\dx
 &\le4\norm{w_{h,R}(t)}_{L^{p_d}}^2
       |A_{h,R}(t)|^{1-2/p_d}\le C R^{-\alpha_d}\norm{\nabla\theta_h(t)}_{L^2}^2,
 \qquad \alpha_d=1-\frac2{p_d}>0.
\end{align*}
After integration in time the right-hand side tends to zero uniformly as
$R\to\infty$.  Hence the family $\{\theta_{h,\tau}^2\}$ is uniformly
integrable.  Since $\theta_{h,\tau}\to\theta$ in measure and
$\theta\in L^2(\QT)$, Vitali's theorem gives strong $L^2$ convergence of
the forward interpolants.

To identify the backward interpolants, apply
Lemma~\ref{lem:log-monotonicity} directly to consecutive states.  With
$\Delta\theta_h^{n+1}=\theta_h^{n+1}-\theta_h^n$, multiplication by $\tau$
and summation give
\begin{align*}
 4\eps_1\Big\|\sqrt{\theta_{h,\tau}^+}
       -\sqrt{\theta_{h,\tau}^-}\Big\|_{L^2(\QT)}^2
 \le{}&\tau\sum_n
 |(e_h^{n+1}-e_h^n,\Delta\theta_h^{n+1})|+\tau\sum_n
 |(F_1(z_h^{n+1})-F_1(z_h^n),\Delta\theta_h^{n+1})|.
\end{align*}
For the first sum, the identity
$e_h^{n+1}-e_h^n=\tau\deltat e_h^{n+1}$ and Cauchy--Schwarz give
\begin{align*}
 &\tau\sum_n
 |(e_h^{n+1}-e_h^n,\Delta\theta_h^{n+1})|\le
 \tau\bigg(\tau\sum_n\norm{\deltat e_h^{n+1}}_{-1,h}^2\bigg)^{1/2}
 \bigg(\tau\sum_n\norm{\Delta\theta_h^{n+1}}_{H^1}^2\bigg)^{1/2}
 \le C\tau,
\end{align*}
where the last factor is bounded by the forward and backward
$L^2(0,T;H^1)$ estimates.  The second sum tends to zero because
$F_1(z_{h,\tau}^+)-F_1(z_{h,\tau}^-)\to0$ strongly in
$L^2(0,T;L^{6/5})$, whereas
$\theta_{h,\tau}^+-\theta_{h,\tau}^-$ is bounded in
$L^2(0,T;L^6)$.  Thus the two square-root interpolants have the same
$L^2$ limit.  It follows that
$\theta_{h,\tau}^+-\theta_{h,\tau}^-\to0$ in $L^1$ and hence in measure.
The tail estimate above applies equally to the backward states, so their
squares are uniformly integrable.  Vitali's theorem finally yields the
strong $L^2$ convergence of $\theta_{h,\tau}^-$.
\end{proof}

\subsection{Thermal nonlinearities and the internal-energy defect}
\label{sec:defect}

Only the highest inverse power can fail to be uniformly integrable.  We
therefore split
\begin{equation}
 e_\sing(\theta)=\eps_q\theta^{-q},
 \qquad
 e_\reg(z,\theta)=\theta^{-1}-\eps_1\ln\theta+F_1(z).
 \label{eq:energy-split}
\end{equation}

\begin{lemma}[Strong convergence of all lower-order thermal terms]
\label{lem:thermal-transforms}
For every $1\le r<q$,
\begin{equation}
 \theta_{h,\tau}^{-1}\to\theta^{-1}
 \quad\text{strongly in }L^r(\QT).
 \label{eq:reciprocal-strong}
\end{equation}
Moreover,
\begin{align}
 (\ln\theta_{h,\tau})^+&\to(\ln\theta)^+
 &&\text{strongly in }L^2(\QT),
 \label{eq:positive-log-strong}\\
 (-\ln\theta_{h,\tau})^+&\to(-\ln\theta)^+
 &&\text{strongly in }L^1(\QT),
 \label{eq:negative-log-strong}\\
 \theta_{h,\tau}^{1-q}&\to\theta^{1-q}
 &&\text{strongly in }L^1(\QT).
 \label{eq:entropy-power-strong}
\end{align}
Consequently,
\begin{align}
 e_\reg(z_{h,\tau},\theta_{h,\tau})&\to e_\reg(z,\theta),
 \label{eq:regular-energy-strong}\\
 \Sbulk(z_{h,\tau},\theta_{h,\tau})&\to\Sbulk(z,\theta)
 \label{eq:bulk-entropy-strong}
\end{align}
strongly in $L^1(\QT)$, and the same conclusions hold for the backward
interpolants.
\end{lemma}

\begin{proof}
Theorem~\ref{thm:temperature-compactness} gives convergence almost
everywhere after extraction.  Since
$\theta_{h,\tau}^{-1}$ is bounded in $L^q(\QT)$, the family
$|\theta_{h,\tau}^{-1}|^r$ is uniformly integrable for every $r<q$.
Vitali's theorem proves \eqref{eq:reciprocal-strong}.

The map $s\mapsto(\ln s)^+$ is one-Lipschitz on $(0,\infty)$, and hence
\eqref{eq:positive-log-strong} follows from the strong $L^2$ convergence of
the temperatures.  For the negative logarithm choose $p>1$.  Since
$[(-\ln s)^+]^p\le C_{p,q}(1+s^{-q})$, the negative logarithms are bounded
in $L^p(\QT)$ and therefore uniformly integrable.  Their almost-everywhere
convergence and Vitali's theorem give \eqref{eq:negative-log-strong}.
Finally,
$$
 \|\theta_{h,\tau}^{1-q}\|_{L^{q/(q-1)}(\QT)}^{q/(q-1)}
 =\int_{\QT}\theta_{h,\tau}^{-q}\dx\dtm\le C.
$$
Because $q/(q-1)>1$, the family is uniformly integrable, which proves
\eqref{eq:entropy-power-strong}.

The phase variables converge strongly in $L^4(\QT)$, so the quartic
polynomials $F_0(z_{h,\tau})$ and $F_1(z_{h,\tau})$ converge strongly in
$L^1(\QT)$.  Combining this fact with the explicit formulas for
$e_\reg$ and $\Sbulk$ proves
\eqref{eq:regular-energy-strong}--\eqref{eq:bulk-entropy-strong}.  The same
argument applies to the backward interpolants by
Theorem~\ref{thm:temperature-compactness} and
Proposition~\ref{prop:phase-compactness}.
\end{proof}

Define the critical singular-energy densities and their associated
time-dependent Radon measures by
\begin{equation}
 g_{h,\tau}^{\pm}
 :=\eps_q(\theta_{h,\tau}^{\pm})^{-q},
 \qquad
 \nu_{h,\tau}^{\pm}(t)
 :=g_{h,\tau}^{\pm}(t,\cdot)\,\mathrm dx
 \in\mathcal M_+(\overline\Omega).
 \label{eq:singular-measure}
\end{equation}
The strong convergence of the inverse temperature rules out an oscillation
defect in this term.  Indeed, for every $M>0$ the bounded continuous
truncation
$$
 b_M(s):=\min\{\eps_q s^{-q},M\},
 \qquad b_M(0):=M,
$$
satisfies
$b_M(\theta_{h,\tau}^{\pm})\to b_M(\theta)$ strongly in $L^1(\QT)$
by almost-everywhere convergence and dominated convergence.  Hence the
ordinary Young measure is the Dirac mass at $\theta$, and the only possible
loss of compactness in the untruncated critical energy is concentration, in
the sense of the concentration defects discussed in
\cite[Sect.~5.1.1]{FeireislLukacovaMizerovaShe2021}.

\begin{lemma}[Critical-energy concentration defect]
\label{lem:defect-construction}
After passing to a further subsequence, there exists
$$
 \lambda\in L^\infty_{\rm w*}
 (0,T;\mathcal M_+(\overline\Omega))
$$
such that, for both temporal interpolants,
\begin{equation}
 \nu_{h,\tau}^{\pm}
 \stackrel{*}{\rightharpoonup}
 \eps_q\theta^{-q}\,\mathrm dx+\lambda
 \quad\text{in }
 L^\infty_{\rm w*}(0,T;\mathcal M(\overline\Omega)).
 \label{eq:defect-measure-valued}
\end{equation}
In particular, the associated space--time measures satisfy
\begin{equation}
 g_{h,\tau}^{\pm}\,\mathrm dx\,\mathrm dt
 \stackrel{*}{\rightharpoonup}
 \eps_q\theta^{-q}\,\mathrm dx\,\mathrm dt
 +\lambda_t\,\mathrm dt
 \quad\text{in }\mathcal M(\overline\Omega\times[0,T]).
 \label{eq:defect-measure}
\end{equation}
Moreover,
\begin{equation}
 \theta^{-q}\in L^\infty(0,T;L^1(\Omega)),
 \qquad
 \operatorname*{ess\,sup}_{t\in(0,T)}
 \norm{\lambda_t}_{\mathcal M(\overline\Omega)}
 \le C.
 \label{eq:defect-fibre-bounds}
\end{equation}
The positive representative $t\mapsto\lambda_t$ is unique up to a
Lebesgue-null set, and the associated space--time measure
$\lambda_t\,\mathrm dt$ has no atoms in the time variable.
\end{lemma}

\begin{proof}
The discrete availability estimate gives, for both interpolants,
\begin{equation}
 \operatorname*{ess\,sup}_{t\in(0,T)}
 \norm{\nu_{h,\tau}^{\pm}(t)}_{\mathcal M(\overline\Omega)}
 =
 \operatorname*{ess\,sup}_{t\in(0,T)}
 \int_\Omega g_{h,\tau}^{\pm}(t)\dx
 \le C.
 \label{eq:singular-measure-Linfty-bound}
\end{equation}
We first use the forward interpolant.  By the weak-star
measure-valued duality introduced in
Subsection~\ref{sec:theoretical-tools},
$\nu_{h,\tau}^{+}$ defines a bounded positive functional on
$L^1(0,T;C(\overline\Omega))$.  Since this predual is separable,
the weak-star topology is metrizable on bounded subsets of its dual;
Banach--Alaoglu therefore yields, after extraction,
\begin{equation}
 \nu_{h,\tau}^{+}\stackrel{*}{\rightharpoonup}\nu
 \quad\text{in }
 L^\infty_{\rm w*}(0,T;\mathcal M(\overline\Omega))
 \label{eq:critical-measure-prelimit}
\end{equation}
for some
$\nu\in L^\infty_{\rm w*}
 (0,T;\mathcal M(\overline\Omega))$, with
$\norm{\nu}_{L^\infty_{\rm w*}}\le C$ by weak-star lower
semicontinuity of the dual norm.  To record positivity at the level of the
time fibres, let $\{\phi_j\}_{j\in\mathbb N}$ be dense in the positive
unit ball of $C(\overline\Omega)$.  For every
$\zeta\in C_c(0,T)$, $\zeta\ge0$, weak-star convergence gives
$$
 \int_0^T\zeta(t)\pair{\nu_t}{\phi_j}\dtm
 =\lim_{h,\tau\to0}
 \int_0^T\zeta(t)\pair{\nu_{h,\tau}^+(t)}{\phi_j}\dtm\ge0.
$$
Hence $\pair{\nu_t}{\phi_j}\ge0$ for almost every $t$.  Intersecting
the corresponding full-measure sets and using density and the uniform dual
norm bound yields
$$
 \nu_t\in\mathcal M_+(\overline\Omega),
 \qquad
 \nu_t(\overline\Omega)
 =\norm{\nu_t}_{\mathcal M(\overline\Omega)}\le C
 \quad\text{for almost every }t.
$$
This is the same compactness mechanism used for time-dependent energy
concentration measures in
\cite[Sect.~5.1.3.2]{FeireislLukacovaMizerovaShe2021}.

Extend $s\mapsto s^{-q}$ by $+\infty$ at $s=0$. By
Theorem~\ref{thm:temperature-compactness},
$\theta_{h,\tau}^{+}\to\theta$ almost everywhere in $\QT$. Hence,
for every non-negative
$\Phi\in C([0,T]\times\overline\Omega)$, Fatou's lemma and
\eqref{eq:critical-measure-prelimit} give
\begin{equation}
 \int_0^T\!\!\int_\Omega
 \eps_q\theta^{-q}\Phi\dx\dtm
 \le
 \int_0^T\pair{\nu_t}{\Phi(t)}\dtm.
 \label{eq:critical-measure-Fatou}
\end{equation}
Taking $\Phi(t,x)=\zeta(t)$, with
$\zeta\in C_c(0,T)$, $\zeta\ge0$, yields
$$
 \int_0^T\zeta(t)
 \left(\eps_q\int_\Omega\theta(t)^{-q}\dx\right)\dtm
 \le
 \int_0^T\zeta(t)\nu_t(\overline\Omega)\dtm
 \le C\int_0^T\zeta(t)\dtm.
$$
Therefore
\begin{equation}
 \eps_q\int_\Omega\theta(t)^{-q}\dx\le C
 \quad\text{for almost every }t\in(0,T),
 \label{eq:limit-critical-fibre-bound}
\end{equation}
which proves the first assertion in
\eqref{eq:defect-fibre-bounds}. In particular, Fubini's theorem shows that the map
$$
 \mu_t:=\eps_q\theta(t)^{-q}\,\mathrm dx
$$
is weak-star measurable and belongs to
$L^\infty_{\rm w*}(0,T;\mathcal M_+(\overline\Omega))$.

Set $\Lambda:=\nu-\mu$. We now verify positivity of its fibres rather
than merely positivity of the associated space--time functional. Let
$\{\phi_j\}_{j\in\mathbb N}$ be a countable dense subset of the positive
unit ball of $C(\overline\Omega)$. Testing
\eqref{eq:critical-measure-Fatou} with
$\Phi(t,x)=\zeta(t)\phi_j(x)$, where
$\zeta\in C_c(0,T)$, $\zeta\ge0$, shows that
$$
 \pair{\Lambda_t}{\phi_j}\ge0
 \quad\text{for almost every }t.
$$
After intersecting these sets with the full-measure set on which the
measure norms are bounded, the inequality holds for every $j$ on one common
full-measure set.  Since
$\norm{\Lambda_t}_{\mathcal M}\le 2C$ there, approximation of the
normalised function $\phi/\norm{\phi}_{C(\overline\Omega)}$ by the dense
family extends the inequality to every
$\phi\in C(\overline\Omega)$, $\phi\ge0$. Thus
$\Lambda_t\in\mathcal M_+(\overline\Omega)$ for almost every $t$.
Define $\lambda:=\Lambda$. Since
$\nu_t=\mu_t+\lambda_t$ and both terms are non-negative,
$$
 \norm{\lambda_t}_{\mathcal M}
 =\lambda_t(\overline\Omega)
 \le\nu_t(\overline\Omega)\le C
 \quad\text{for almost every }t,
$$
which proves the remaining estimate in
\eqref{eq:defect-fibre-bounds} and the forward convergence in
\eqref{eq:defect-measure-valued}.

It remains to identify the backward interpolant. Put
$$
 \nu_h^n:=\eps_q(\theta_h^n)^{-q}\,\mathrm dx,
 \qquad n=0,\ldots,N.
$$
For
$\Phi\in C^1([0,T];C(\overline\Omega))$, a direct shift of the time-cell
sum gives
\begin{align}
 &\int_0^T
 \pair{\nu_{h,\tau}^{+}(t)-\nu_{h,\tau}^{-}(t)}{\Phi(t)}\dtm
 \nonumber\\
 &\quad=
 \sum_{n=1}^{N-1}\int_{t^{n-1}}^{t^n}
 \pair{\nu_h^n}{\Phi(t)-\Phi(t+\tau)}\dtm
 +\int_{t^{N-1}}^T\pair{\nu_h^N}{\Phi(t)}\dtm
 -\int_0^\tau\pair{\nu_h^0}{\Phi(t)}\dtm.
 \label{eq:forward-backward-measure-shift}
\end{align}
Using the uniform bound \eqref{eq:singular-measure-Linfty-bound},
\begin{equation}
 \left|
 \int_0^T
 \pair{\nu_{h,\tau}^{+}-\nu_{h,\tau}^{-}}{\Phi}\dtm
 \right|
 \le
 C\tau\left(
 \norm{\partial_t\Phi}_{L^\infty(0,T;C(\overline\Omega))}
 +\norm{\Phi}_{L^\infty(0,T;C(\overline\Omega))}
 \right).
 \label{eq:forward-backward-measure-bound}
\end{equation}
Banach-valued temporal mollification shows that
$C^1([0,T];C(\overline\Omega))$ is dense in
$L^1(0,T;C(\overline\Omega))$.  For a general test $\Psi$ in the latter
space, approximate it by $\Phi_j$ in the former space.  The uniform
$L^\infty_{\rm w*}$ bound controls the error by
$2C\norm{\Psi-\Phi_j}_{L^1(0,T;C)}$, whereas
\eqref{eq:forward-backward-measure-bound} tends to zero for each fixed
$\Phi_j$.  First letting $(h,\tau)\to0$ and then $j\to\infty$ proves that
$\nu_{h,\tau}^{+}-\nu_{h,\tau}^{-}$ converges weak-star to zero in
$L^\infty_{\rm w*}(0,T;\mathcal M(\overline\Omega))$. Thus the backward interpolant has the same limit.

Finally, the positive bounded functional
$$
 \Phi\longmapsto
 \int_0^T\pair{\lambda_t}{\Phi(t)}\dtm,
 \qquad
 \Phi\in C(\overline\Omega\times[0,T]),
$$
defines, by the Riesz representation theorem, a finite non-negative Radon
measure $\boldsymbol\lambda$ on
$\overline\Omega\times[0,T]$.  Its time marginal is absolutely
continuous.  Indeed, for every $\zeta\in C([0,T])$,
$$
 \int_{[0,T]}\zeta(t)\,
 \mathrm d(\pi_{t\#}\boldsymbol\lambda)(t)
 =\int_0^T\zeta(t)\lambda_t(\overline\Omega)\dtm.
$$
Uniqueness in the Riesz representation theorem therefore gives
$$
 \boldsymbol\lambda(\overline\Omega\times B)
 =\int_B\lambda_t(\overline\Omega)\dtm\le C|B|
 \qquad\text{for every Borel }B\subset[0,T].
$$
Consequently, $\boldsymbol\lambda$ has no atoms in the time variable, and
\eqref{eq:defect-measure} follows from
\eqref{eq:defect-measure-valued}.
\end{proof}

\begin{theorem}[Convergence of subsequences]\label{thm:conv_to_limit}
     For every $h>0$ and every time step $\tau=T/N$, the fully discrete problem
of Subsection~\ref{sec:scheme} admits at least one global discrete trajectory
such that
\begin{equation}
 \theta_h^n>0\quad\text{on }\overline\Omega,
 \qquad n=0,\ldots,N.
 \label{eq:main-discrete-positivity}
\end{equation}
Moreover, let $(h_k,\tau_k)\to(0,0)$ be arbitrary, without a coupling
condition, and choose at every time step any solution furnished by
Theorem~\ref{thm:one-step-existence}.  Then there exist a subsequence, not
relabeled, and \aaron{functions}
$(\rho,\mu_\rho,\theta,\eta,\mu_\eta,\lambda)$ such that
\begin{align}
 z_{h_k,\tau_k}^{\pm}&\to z
 &&\text{strongly in }L^p(\QT;\mathbb R^2),
 &&1\le p<6,
 \label{eq:main-phase-convergence}\\
 \theta_{h_k,\tau_k}^{\pm}&\to\theta
 &&\text{strongly in }L^2(\QT),
 \label{eq:main-temperature-convergence}\\
 \mu_{\rho,h_k,\tau_k}&\rightharpoonup\mu_\rho
 &&\text{weakly in }L^2(0,T;H^1(\Omega)),
 \label{eq:main-murho-convergence}\\
 \nabla\theta_{h_k,\tau_k}&\rightharpoonup\nabla\theta
 &&\text{weakly in }L^2(\QT;\mathbb R^d),
 \label{eq:main-gradient-temperature-convergence}\\
 \mu_{\eta,h_k,\tau_k}&\rightharpoonup\mu_\eta
 &&\text{weakly in }L^2(\QT).
 \label{eq:main-mueta-convergence}
\end{align}
The critical singular energies satisfy
\begin{equation}
\eps_q(\theta_{h_k,\tau_k}^{\pm})^{-q}\,\mathrm dx
 \stackrel{*}{\rightharpoonup}
 \eps_q\theta^{-q}\,\mathrm dx+\lambda
 \quad\text{in }
 L^\infty_{\rm w*}(0,T;\mathcal M(\overline\Omega)),
 \label{eq:main-defect-convergence}
\end{equation}
where
$\lambda\in
L^\infty_{\rm w*}(0,T;\mathcal M_+(\overline\Omega))$.
In particular, the corresponding space--time measures converge as in
\eqref{eq:defect-measure}.
\end{theorem}

\begin{proof}
Corollary~\ref{cor:global-discrete-solutions} gives a positive global
discrete trajectory for every pair $(h,\tau)$. The uniform estimates of
Theorem~\ref{thm:uniform-estimates} imply, after extraction,
$$
 \mu_{\rho,h,\tau}\rightharpoonup\mu_\rho
 \quad\text{in }L^2(0,T;H^1(\Omega)),
 \qquad
 \mu_{\eta,h,\tau}\rightharpoonup\mu_\eta
 \quad\text{in }L^2(\QT),
$$
and
$$
 \theta_{h,\tau}\rightharpoonup\theta
 \quad\text{in }L^2(0,T;H^1(\Omega)).
$$
Proposition~\ref{prop:phase-compactness} yields the strong phase
convergences in \eqref{eq:main-phase-convergence}, while
Theorem~\ref{thm:temperature-compactness} yields
\eqref{eq:main-temperature-convergence} for both temporal interpolants.
The weak limit of the gradients is therefore $\nabla\theta$, which gives
\eqref{eq:main-gradient-temperature-convergence}.
Finally, Lemma~\ref{lem:defect-construction} supplies
\eqref{eq:main-defect-convergence}, after passing to one further
subsequence. Since only finitely many extractions are involved, all
convergences hold along one common subsequence.
\end{proof}

\begin{lemma}[Time representatives and product rules]
\label{lem:time-representatives}
Every generalised dissipative weak solution has representatives satisfying
\begin{equation}
 \rho,\eta\in C([0,T];L^2(\Omega)),
 \qquad
 \mathfrak E\in C_{\rm w*}([0,T];\mathcal M(\overline\Omega)).
 \label{eq:time-representatives}
\end{equation}
More precisely, for every integer $m>d/2$,
\begin{equation}
 \mathfrak E\in W^{1,2}(0,T;H^{-m}(\Omega))
 \hookrightarrow C([0,T];H^{-m}(\Omega)),
 \qquad
 \partial_t\mathfrak E=\diver\mathcal J_e.
 \label{eq:energy-measure-time-regularity}
\end{equation}
The initial trace is
$\mathfrak E_0=e(z_0,\theta_0)\,\mathrm dx$.  In particular, for all
$t\in[0,T]$ and all sufficiently regular time-dependent periodic tests,
\begin{align}
 (\rho(t),v(t))-(\rho_0,v(0))
 &=\int_0^t\bigl[(\rho,\partial_s v)
       -(\mathcal J_\rho,\nabla v)\bigr]\,\mathrm ds,
 \label{eq:rho-product-rule}\\
 \pair{\mathfrak E_t}{\xi(t)}
 -\int_\Omega e(z_0,\theta_0)\xi(0)\dx
 &=\int_0^t\bigl[\pair{\mathfrak E_s}{\partial_s\xi(s)}
       -(\mathcal J_e,\nabla\xi)\bigr]\,\mathrm ds,
 \label{eq:energy-product-rule}\\
 (\eta(t),w(t))-(\eta_0,w(0))
 &=\int_0^t\bigl[(\eta,\partial_s w)
       -(\mathcal R_\eta,w)\bigr]\,\mathrm ds.
 \label{eq:eta-product-rule}
\end{align}
Here $v$ may be taken in
$W^{1,1}(0,T;L^2)\cap L^2(0,T;H^1)$ and $w$ in
$W^{1,1}(0,T;L^2)$.  For the energy identity it is sufficient that
$\xi\in W^{1,1}(0,T;W^{1,\infty})$.
\end{lemma}

\begin{proof}
The flux bounds imply
$\partial_t\rho=\diver\mathcal J_\rho\in L^2(0,T;H^{-1})$ and
$\partial_t\eta=-\mathcal R_\eta\in L^2(\QT)$.  Together with the state
bounds, the standard Hilbert-space evolution theorem gives the first two
representatives in \eqref{eq:time-representatives} and the corresponding
product rules.

The growth estimate \eqref{eq:e-s-growth}, the weak-solution bounds, and the
uniform bound on $\lambda_t$ show that $\mathfrak E$ is essentially bounded
in $\mathcal M(\overline\Omega)$.  Since finite measures embed continuously
into $H^{-m}(\Omega)$ for every integer $m>d/2$, the weak energy equation
and $\mathcal J_e\in L^2(\QT)$ imply
\eqref{eq:energy-measure-time-regularity}.  The weak formulation identifies
its value at $t=0$.  The $H^{-m}$-continuous representative is also weak-star
continuous as a measure: approximate a continuous spatial test uniformly by
smooth functions and use the uniform total-variation bound.  Formula
\eqref{eq:energy-product-rule} follows first for smooth tests.  For the
stated Lipschitz class, use periodic spatial mollification and then time
mollification: the test functions and their time derivatives converge
uniformly, while the spatial gradients converge almost everywhere and stay
uniformly bounded.  The measure and $L^2$ flux pairings therefore pass to
the limit by bounded and dominated convergence.  The phase product rules
follow by the usual Hilbert-space density argument.
\end{proof}

\begin{lemma}[Identification of the split phase derivatives]
\label{lem:split-identification}
For $\sigma\in\{\rho,\eta\}$,
\begin{align}
 &\partial_\sigma\psi_\vex(z_{h,\tau},\theta_{h,\tau})
 +\partial_\sigma\psi_\cav(z_{h,\tau}^-,\theta_{h,\tau})\to\partial_\sigma\psi(z,\theta)
 \qquad\text{strongly in }L^1(\QT).
 \label{eq:split-identification}
\end{align}
\end{lemma}

\begin{proof}
The arguments converge almost everywhere, and
\eqref{eq:split-growth} bounds the expressions by
$$
 C(1+\theta_{h,\tau})
 (1+|z_{h,\tau}|^3+|z_{h,\tau}^-|^3).
$$
This family is bounded in $L^{3/2}(\QT)$ by the
$L^2(0,T;L^6)$ bound for the temperature and the
$L^\infty(0,T;L^6)$ phase bounds.  It is therefore uniformly integrable,
and Vitali's theorem proves the claim.  The forward and backward phases
have the same limit, so the split pieces recombine into
$\partial_\sigma\psi(z,\theta)$.
\end{proof}

\subsection{Passage to the limit and proof to the main theorem}
\label{sec:limit}
With these ingredients, we can proceed to prove our main result, i.e. Theorem~\ref{thm:convergence}.

\begin{proof}[Proof of Theorem~\ref{thm:convergence}]
The existence of a global positive discrete trajectory, together with the weak and strong compactness statements needed below, is provided by Theorem~\ref{thm:conv_to_limit}.  Since the old-state
variables converge almost everywhere and $\LL$ is bounded and continuous,
\begin{equation}
 \LL(z_{h,\tau}^-,\theta_{h,\tau}^-)
 \to\LL(z,\theta)
 \qquad\text{strongly in }L^p(\QT)
 \quad\text{for every finite }p.
 \label{eq:mobility-strong}
\end{equation}
If $A_{h,\tau}$ denotes any block of this matrix and
$v_{h,\tau}\rightharpoonup v$ in $L^2$, then
$A_{h,\tau}v_{h,\tau}\rightharpoonup Av$ in distributions, because
$A_{h,\tau}^{\top}\phi\to A^{\top}\phi$ strongly in $L^2$ for every
bounded test function $\phi$.  This identifies all three limiting Onsager
expressions.

We spell out the discrete summation by parts for the energy equation.  Let
$\xi\in C^1([0,T];C^\infty_{\rm per}(\Omega))$ satisfy $\xi(T)=0$, set
$\xi_h^n=\Ih\xi(t^n)$, and use $\xi_h^{n+1}$ in
\eqref{eq:fd-energy}.  Multiplication by $\tau$ and summation over
$n=0,\ldots,N-1$ give the exact identity
\begin{equation}
 -\int_0^T(e_{h,\tau}^-,\partial_t\widetilde\xi_{h,\tau})\dtm
 -(e_h^0,\xi_h^0)
 +\int_0^T(J_{e,h,\tau},\nabla\xi_{h,\tau}^+)\dtm=0,
 \label{eq:discrete-energy-tested}
\end{equation}
where
$$
 J_{e,h,\tau}
 =(\LL_{12,h,\tau}^-)^\top\nabla\mu_{\rho,h,\tau}
 -\LL_{22,h,\tau}^-\nabla\theta_{h,\tau}
 +\LL_{23,h,\tau}^-\mu_{\eta,h,\tau}.
$$
The regular part internal energy converges strongly in $L^1$, and
Lemma~\ref{lem:defect-construction} gives the singular limit for the
backward interpolant.  Since
$\partial_t\widetilde\xi_{h,\tau}\to\partial_t\xi$ and
$\nabla\xi_{h,\tau}^+\to\nabla\xi$ uniformly,
\eqref{eq:discrete-energy-tested} converges to
\eqref{eq:weak-energy}.  Initial convergence follows from
Lemma~\ref{lem:initial-approximation}.  The two phase balance laws follow
from the identical discrete summation-by-parts formula, using the strong
initial convergence and the flux identification above.

For the constitutive equations, use finite-element interpolants of smooth
space--time tests and integrate over time  Weak convergence handles the gradient terms, while
Lemma~\ref{lem:split-identification} handles the nonlinear terms.  Thus
\eqref{eq:weak-murho}--\eqref{eq:weak-mueta} hold.

It remains to pass to the entropy inequality.  Set
$Z_{h,\tau}=(\nabla\mu_{\rho,h,\tau},-\nabla\theta_{h,\tau},
\mu_{\eta,h,\tau})$.  The positive square roots of the mobility matrices
converge almost everywhere and strongly in every finite $L^p$.  Therefore,
for each $t\in(0,T)$,
\begin{equation}
 \int_0^t\!\int_\Omega Z^\top\LL(z,\theta)Z\dx\,\mathrm ds
 \le\liminf_{h,\tau\to0}
 \int_0^t\!\int_\Omega
 Z_{h,\tau}^\top\LL_h^-Z_{h,\tau}\dx\,\mathrm ds.
 \label{eq:dissipation-lsc}
\end{equation}
By \eqref{eq:bulk-entropy-strong} and
Lemma~\ref{lem:time-slice-extraction}, one common subsequence satisfies, for
almost every $t$,
$$
 \Sbulk(z_{h,\tau}^+(t),\theta_{h,\tau}^+(t))
 \to\Sbulk(z(t),\theta(t))
 \quad\text{in }L^1(\Omega),
$$
while the phase fields converge strongly in $L^2$ and weakly in $H^1$ at
that time.  Weak lower semicontinuity of the two Dirichlet integrals then
gives
$$
 \Stot(z(t),\theta(t))
 \ge\limsup_{h,\tau\to0}\Stot_h^+(t)
 \quad\text{for almost every }t.
$$
If $t\in(t^n,t^{n+1}]$, the iterated discrete entropy inequality gives
$$
 \Stot_h^+(t)\ge\Stot_h^0
 +\int_0^t\!\int_\Omega
 Z_{h,\tau}^\top\LL_h^-Z_{h,\tau}\dx\,\mathrm ds,
$$
because the omitted part of the last time cell is non-negative.  Combining
the last three displays with convergence of the initial total entropy proves
\eqref{eq:weak-entropy}.  All regularity requirements in
Definition~\ref{def:weak-solution} follow from the uniform estimates and
Lemma~\ref{lem:defect-construction}.


It remains only to record the two conservation laws stated in the theorem.
Lemma~\ref{lem:time-representatives} gives continuous representatives of
$\rho$ and of the total internal-energy measure $\mathfrak E$.
Taking spatially constant test functions in
\eqref{eq:rho-product-rule} and \eqref{eq:energy-product-rule} yields
$$
 \int_\Omega\rho(t)\dx=\int_\Omega\rho_0\dx,
 \qquad
 \mathfrak E_t(\overline\Omega)
 =\int_\Omega e(z_0,\theta_0)\dx
 \quad\text{for every }t\in[0,T].
$$
This completes the proof.
\end{proof}

\subsection{Vanishing defect measure}\label{sec:vanshdefect}

Under suitable information on the inverse temperature one can even prove that the defect measure vanishes.

\begin{proposition}[Defect-removal criteria]
\label{prop:no-defect}
Put $g_{h,\tau}=\eps_q\theta_{h,\tau}^{-q}$.  The defect vanishes if any
one of the following conditions holds:
\begin{enumerate}[(i)]
\item the singular energies are uniformly integrable,
\begin{equation}
 \lim_{M\to\infty}\sup_{h,\tau}
 \int_{\QT} g_{h,\tau}\mathbf 1_{\{g_{h,\tau}>M\}}\dx\dtm=0;
 \label{eq:uniform-integrability}
\end{equation}
\item their space--time masses converge,
\begin{equation}
 \int_{\QT} g_{h,\tau}\dx\dtm
 \longrightarrow\int_{\QT}\eps_q\theta^{-q}\dx\dtm;
 \label{eq:singular-mass-convergence}
\end{equation}
\item for some $\delta>0$,
\begin{equation}
 \sup_{h,\tau}\int_{\QT} g_{h,\tau}^{1+\delta}\dx\dtm<\infty.
 \label{eq:singular-superlinear-bound}
\end{equation}
\end{enumerate}
\end{proposition}

\begin{proof}
Condition~\eqref{eq:uniform-integrability}, almost-everywhere convergence,
and Vitali's theorem give
$g_{h,\tau}\to\eps_q\theta^{-q}$ strongly in $L^1(\QT)$, so the
concentration defect is zero. Under
\eqref{eq:singular-mass-convergence}, test the weak-star convergence
\eqref{eq:defect-measure-valued} with the constant function $1$ to obtain
$$
 \int_0^T\lambda_t(\overline\Omega)\dtm=0.
$$
Since $\lambda_t\ge0$ for almost every $t$, this implies
$\lambda=0$. Finally,
\eqref{eq:singular-superlinear-bound} implies uniform integrability either
by the de la Vall\'ee--Poussin criterion
\cite[Preliminary Material, Thm.~14]{FeireislLukacovaMizerovaShe2021}, or
directly from
$$
 g\mathbf 1_{\{g>M\}}\le M^{-\delta}g^{1+\delta}.
$$
\end{proof}

\section{Weak--strong uniqueness and long-time behavior}
\label{sec:wsu}

The compactness argument gives integral control but no mesh-uniform
pointwise separation of the inverse temperature.  We therefore formulate
the stability theorem on a bounded state range.  Under this additional
hypothesis, all entropy-variable remainders are genuinely quadratic, and the
state-dependent mobility can be treated without suppressing any terms.

\subsection{Relative entropy and bounded-range coercivity}
\label{sec:relative-coercivity}

Because $\partial_\theta e<0$ and
$e(z,\theta)\to+\infty$ as $\theta\downarrow0$, whereas
$e(z,\theta)\to-\infty$ as $\theta\to\infty$, the map
$\theta\mapsto e(z,\theta)$ is a $C^2$ bijection from $(0,\infty)$ onto
$\mathbb R$ for every fixed $z$.  We denote its inverse by
$$
 \theta=\Theta(z,\varepsilon),
 \qquad \varepsilon=e(z,\theta),
$$
and introduce the bulk entropy in the energy variables,
$$
 \bar s(\rho,\varepsilon,\eta)
 :=\Sbulk\bigl((\rho,\eta),
       \Theta((\rho,\eta),\varepsilon)\bigr).
$$
The Gibbs identity gives
\begin{equation}
 D_{(\rho,\varepsilon,\eta)}\bar s
 =\bigl(-\partial_\rho\psi,\Theta,-\partial_\eta\psi\bigr).
 \label{eq:entropy-gradient-energy-variables}
\end{equation}
We write
$$
 U=(\rho,e(z,\theta),\eta),
 \qquad
 \widehat U=(\widehat\rho,
 e(\widehat z,\widehat\theta),\widehat\eta).
$$

\begin{definition}[Bounded-range weak solution]
\label{def:bounded-range-weak}
A generalised dissipative weak solution is called $M$-bounded if
\begin{equation}
 |z|\le M,
 \qquad M^{-1}\le\theta\le M
 \quad\text{almost everywhere in }\QT.
 \label{eq:weak-bounded-range}
\end{equation}
\end{definition}

\begin{definition}[Strong comparison solution]
\label{def:strong-comparison}
A strong comparison solution is a defect-free solution
$(\widehat\rho,\widehat\mu_\rho,\widehat\theta,
\widehat\eta,\widehat\mu_\eta)$ satisfying the equations almost everywhere,
the corresponding entropy equality, and the following sufficient
regularity conditions:
\begin{align}
 &\widehat U\in W^{1,1}(0,T;L^\infty(\Omega)^3),
 \qquad
 \widehat z\in W^{1,1}(0,T;H^1(\Omega)^2),
 \label{eq:strong-statevector-regularity}\\
 &\widehat\theta\in W^{1,1}(0,T;W^{1,\infty}(\Omega)),
 \qquad
 \widehat z\in L^\infty(0,T;W^{1,\infty}(\Omega)^2),
 \label{eq:strong-state-spatial-regularity}\\
 &\widehat\mu_\rho
 \in W^{1,1}(0,T;H^1(\Omega))
       \cap L^\infty(0,T;W^{1,\infty}(\Omega)),
 \label{eq:strong-murho-regularity}\\
 &\widehat\mu_\eta
 \in W^{1,1}(0,T;L^2(\Omega))\cap L^\infty(\QT).
 \label{eq:strong-mueta-regularity}
\end{align}
We also assume that, for some $M>1$,
\begin{equation}
 |\widehat z|\le M,
 \qquad M^{-1}\le\widehat\theta\le M
 \quad\text{on }\overline\Omega\times[0,T].
 \label{eq:strong-bounded-range}
\end{equation}
These assumptions are not intended to be optimal; they ensure that all
strong entropy variables are admissible in the weak formulation and that
pairings with the defect measure are classical.
\end{definition}

For a reference state, define
\begin{align}
 \mathcal H_0(U\mid\widehat U)
 :={}&\int_\Omega
 \bigl[\bar s(\widehat U)-\bar s(U)
       +D\bar s(\widehat U)\cdot(U-\widehat U)\bigr]\dx
+\frac{\gamma_\rho}{2}
 \norm{\nabla(\rho-\widehat\rho)}_{L^2}^2
 +\frac{\gamma_\eta}{2}
 \norm{\nabla(\eta-\widehat\eta)}_{L^2}^2,
 \label{eq:relative-entropy-energy-variables}\\
 \mathcal H_\alpha(U\mid\widehat U)
 :={}&\mathcal H_0(U\mid\widehat U)
 +\frac\alpha2\norm{\rho-\widehat\rho}_{L^2}^2
 +\frac\alpha2\norm{\eta-\widehat\eta}_{L^2}^2.
 \label{eq:relative-entropy-penalized}
\end{align}
The bulk term is the Bregman remainder of the negative entropy in energy
variables.  Equivalently, in the original variables it is the Bregman
remainder of the availability
$\widehat\theta e-\Sbulk$ with $a=\widehat\theta$.
There is no term linear in $\theta-\widehat\theta$, because
$\partial_\theta W_{\widehat\theta}
 (\widehat z,\widehat\theta)=0$.

 \begin{remark}
The same relative entropy argument in original variables is considered in \cite{BrunkHabrichOyedejiYangXu}, which indicates that the assumption necessary for the transformation between temperature and internal energy can be weakened.
 \end{remark}

\begin{lemma}[Bounded-range coercivity]
\label{lem:relative-coercivity}
Fix $M>1$.  There are $\alpha_M>0$ and $c_M>0$ such that, whenever
both states satisfy \eqref{eq:weak-bounded-range} and
\eqref{eq:strong-bounded-range}, every fixed $\alpha\ge\alpha_M$ admits a
constant $C_{M,\alpha}>0$ for which
\begin{align}
 c_M\bigl(&\norm{z-\widehat z}_{H^1}^2
           +\norm{\theta-\widehat\theta}_{L^2}^2\bigr)
 \le\mathcal H_\alpha(U\mid\widehat U)
\le C_{M,\alpha}\bigl(\norm{z-\widehat z}_{H^1}^2
           +\norm{\theta-\widehat\theta}_{L^2}^2\bigr).
 \label{eq:relative-coercivity}
\end{align}
In particular, with a constant depending only on the compact state range,
\begin{equation}
 \norm{U-\widehat U}_{L^2}^2
 +\norm{z-\widehat z}_{H^1}^2
 \le C_M\mathcal H_\alpha(U\mid\widehat U).
 \label{eq:statevector-difference-controlled}
\end{equation}
\end{lemma}

\begin{proof}
Let
$$
 K_M=\{(z,e(z,\theta)):|z|\le M,
             \ M^{-1}\le\theta\le M\}.
$$
The convex hull of $K_M$ is compact in the energy variables.  Since the
inverse map $(z,\varepsilon)\mapsto(z,\Theta(z,\varepsilon))$ is smooth,
all first and second derivatives used below are bounded on a compact
neighbourhood of this convex hull.  Moreover,
\begin{equation}
 \partial_{\varepsilon\varepsilon}\bar s
 =\partial_\varepsilon\Theta
 =\frac1{\partial_\theta e}<0,
 \label{eq:entropy-ee}
\end{equation}
and $-\partial_{\varepsilon\varepsilon}\bar s$ is bounded above and away
from zero on that neighbourhood.

In block form, the Hessian of the penalised bulk integrand is
$$
 \begin{pmatrix}
  -D_{zz}^2\bar s+\alpha I&-D_{z\varepsilon}^2\bar s\\
  -D_{\varepsilon z}^2\bar s&
  -\partial_{\varepsilon\varepsilon}\bar s
 \end{pmatrix}.
$$
The lower-right block is uniformly positive.  All other blocks are bounded,
so choosing $\alpha\ge\alpha_M$ makes the Schur complement uniformly
positive.  Taylor's formula with integral remainder along the segment
between $U$ and $\widehat U$ therefore gives a lower bound with a constant
uniform for $\alpha\ge\alpha_M$ and, for each fixed $\alpha$, an upper
bound by a constant depending on $M$ and $\alpha$, both times multiplied by
$|z-\widehat z|^2+|\varepsilon-\widehat\varepsilon|^2$.
The change of variables is bi-Lipschitz on the compact range, hence this
quantity is equivalent to
$|z-\widehat z|^2+|\theta-\widehat\theta|^2$.
Adding the exact positive gradient remainders proves
\eqref{eq:relative-coercivity} and \eqref{eq:statevector-difference-controlled}.
\end{proof}

\subsection{Relative-entropy inequality with defect}
\label{sec:relative-inequality}

Set
$$
 Y(U)=D\bar s(U),
 \qquad
 \mathcal R_Y(U\mid\widehat U)
 =Y(U)-Y(\widehat U)-DY(\widehat U)(U-\widehat U).
$$
The constitutive functions are smooth and their derivatives through third
order are bounded on the compact range.  Taylor's formula therefore gives
\begin{equation}
 |\mathcal R_Y(U\mid\widehat U)|
 \le C_M|U-\widehat U|^2.
 \label{eq:entropy-variable-remainder}
\end{equation}
For notational brevity, $\LL(U)$ below means
$\LL(z,\Theta(z,\varepsilon))$.

\begin{lemma}[Relative-entropy inequality with defect]
\label{lem:relative-exergy-inequality}
Let an $M$-bounded generalised dissipative weak solution be given, and let
$U=(\rho,e(z,\theta),\eta)$ denote its energy-state vector.  Let
$\widehat U$ be the energy-state vector of a strong comparison solution on
the same bounded range, and assume that $\LL$ is $C^1$ on a neighbourhood
of that range.  Then, for
almost every $t\in(0,T)$,
\begin{align}
 &\mathcal H_0(t)
 +\pair{\lambda_t}{\widehat\theta(t)}
 +\int_0^t\!\int_\Omega
 (Z-\widehat Z)^\top\LL(U)(Z-\widehat Z)\dx\,\mathrm ds
 \nonumber\\
 &\quad\le\mathcal H_0(0)
 +\int_0^t\!\int_\Omega
 (Z-\widehat Z)^\top
 \bigl(\LL(\widehat U)-\LL(U)\bigr)\widehat Z
 \dx\,\mathrm ds
 \nonumber\\
 &\qquad
 -\int_0^t\!\int_\Omega
 \mathcal R_Y(U\mid\widehat U)\cdot\partial_t\widehat U
 \dx\,\mathrm ds
 +\int_0^t\pair{\lambda_s}{\partial_t\widehat\theta(s)}\,\mathrm ds.
 \label{eq:relative-exergy-inequality}
\end{align}
Here
$\widehat Z=(\nabla\widehat\mu_\rho,
-\nabla\widehat\theta,\widehat\mu_\eta)$.
\end{lemma}

\begin{proof}
Let $\mathscr S[U]$ denote the total entropy, including the two negative
Dirichlet energies, and let
$$
 y=\frac{\delta\mathscr S}{\delta U}
   =(-\mu_\rho,\theta,-\mu_\eta),
 \qquad
 \widehat y=(-\widehat\mu_\rho,
              \widehat\theta,-\widehat\mu_\eta).
$$
With
\begin{equation}
 \mathcal B(v_1,v_2,v_3)=(\nabla v_1,\nabla v_2,v_3),
 \qquad
 \mathcal B^*(w_1,w_2,w_3)
 =(-\diver w_1,-\diver w_2,w_3),
 \label{eq:B-operator}
\end{equation}
we have $\mathcal B y=-Z$, and the balance laws take the Onsager form
\begin{equation}
 \partial_t U=\mathcal B^*\LL(U)\mathcal B y.
 \label{eq:onsager-operator-form}
\end{equation}
For the weak solution this identity is understood with the energy component
replaced by the measure $\mathfrak E$.
For almost every $t$, define the augmented relative functional through
the continuous total-energy representative by
\begin{align*}
\mathcal G(t)
 :={}&\mathscr S[\widehat U(t)]-\mathscr S[U(t)]
 +(-\widehat\mu_\rho(t),\rho(t)-\widehat\rho(t))\\
 &-(\widehat\mu_\eta(t),\eta(t)-\widehat\eta(t))
 +\pair{\mathfrak E_t}{\widehat\theta(t)}-\int_\Omega e(\widehat z(t),\widehat\theta(t))
       \widehat\theta(t)\,\mathrm dx.
\end{align*}
Using
$\mathfrak E_t=e(z(t),\theta(t))\,\mathrm dx+\lambda_t$ for almost every
$t$, this is precisely
$$
 \mathcal G(t)=\mathcal H_0(t)
 +\pair{\lambda_t}{\widehat\theta(t)}.
$$
The value at $t=0$ is understood through
$\mathfrak E_0=e(z_0,\theta_0)\,\mathrm dx$. All measure pairings are
well defined by Lemma~\ref{lem:time-representatives}. Apply the three product rules
\eqref{eq:rho-product-rule}--\eqref{eq:eta-product-rule} with the strong
entropy variables and subtract the corresponding strong identities.  The
regularity in Definition~\ref{def:strong-comparison} permits these tests
directly.  We obtain
\begin{align}
 \mathcal G(t)-\mathcal G(0)
 &\le
 -\int_0^t\!\int_\Omega Z^\top\LL(U)Z\dx\,\mathrm ds
 +\int_0^t\!\int_\Omega
    \widehat Z^\top\LL(U)Z\dx\,\mathrm ds
 \nonumber\\
 &\quad
 +\int_0^t\!\int_\Omega
    \partial_s\widehat y\cdot(U-\widehat U)\dx\,\mathrm ds
 +\int_0^t\pair{\lambda_s}{\partial_s\widehat\theta(s)}\,\mathrm ds.
 \label{eq:integrated-relative-preidentity}
\end{align}
The inequality, rather than equality, occurs only because the weak solution
satisfies an entropy inequality.  In deriving this formula, the strong
entropy production cancels exactly with the term
$-\langle\widehat y,\partial_t\widehat U\rangle$ from the differentiated
cross pairing.

We next transform the term containing $\partial_t\widehat y$.  Write
$Y=D\bar s$ for the bulk entropy variables.  With
$\delta z=z-\widehat z$ and $\delta U=U-\widehat U$, the two sides below
are understood through the identities
\begin{align*}
 \int_\Omega\partial_t\widehat y\cdot\delta U\dx
 ={}&\int_\Omega\partial_tY(\widehat U)\cdot\delta U\dx
 -\gamma_\rho(\nabla\partial_t\widehat\rho,\nabla\delta\rho)
 -\gamma_\eta(\nabla\partial_t\widehat\eta,\nabla\delta\eta),\\
 \int_\Omega(y-\widehat y)\cdot\partial_t\widehat U\dx
 ={}&\int_\Omega(Y(U)-Y(\widehat U))\cdot\partial_t\widehat U\dx
 -\gamma_\rho(\nabla\delta\rho,\nabla\partial_t\widehat\rho)
 -\gamma_\eta(\nabla\delta\eta,\nabla\partial_t\widehat\eta).
\end{align*}
All terms are integrable under Definition~\ref{def:strong-comparison}.
Since
$\partial_tY(\widehat U)=DY(\widehat U)\partial_t\widehat U$ and
$DY=D^2\bar s$ is symmetric, Taylor's formula gives
\begin{equation}
 \int_\Omega\partial_t\widehat y\cdot(U-\widehat U)\dx
 =\int_\Omega(y-\widehat y)\cdot\partial_t\widehat U\dx
 -\int_\Omega\mathcal R_Y(U\mid\widehat U)
       \cdot\partial_t\widehat U\dx.
 \label{eq:hessian-symmetry-step}
\end{equation}
The interfacial terms have no remainder because the Dirichlet part of the
entropy is quadratic.  The first term on the right is then evaluated in
Onsager duality.  Using the strong equation
$\partial_t\widehat U=\mathcal B^*\LL(\widehat U)\mathcal B\widehat y$
and periodic integration by parts,
\begin{equation}
 \int_\Omega(y-\widehat y)\cdot\partial_t\widehat U\dx
 =\int_\Omega
 (Z-\widehat Z)^\top\LL(\widehat U)\widehat Z\dx.
 \label{eq:strong-cross-onsager}
\end{equation}
The last pairing is well defined because
$Z-\widehat Z\in L^2(\QT)$ and $\widehat Z\in L^\infty(\QT)$.

Substituting \eqref{eq:hessian-symmetry-step} and
\eqref{eq:strong-cross-onsager} into
\eqref{eq:integrated-relative-preidentity}, and using
$$
 -Z^\top\LL(U)Z+\widehat Z^\top\LL(U)Z
 +(Z-\widehat Z)^\top\LL(\widehat U)\widehat Z
 =-(Z-\widehat Z)^\top\LL(U)(Z-\widehat Z)
 +(Z-\widehat Z)^\top
   (\LL(\widehat U)-\LL(U))\widehat Z,
$$
gives \eqref{eq:relative-exergy-inequality}.  Finally,
$\mathfrak E_0=e(z_0,\theta_0)\,\mathrm dx$ by
Lemma~\ref{lem:time-representatives}; hence there is no initial defect
contribution.  This completes the rigorous weak--strong calculation.
\end{proof}

\subsection{Proof of weak--strong uniqueness}
\label{sec:wsu-proof}

\begin{theorem}[Bounded-range WSU]
\label{thm:bounded-range-wsu}
Let an $M$-bounded generalised dissipative weak solution and a strong
comparison solution have the same initial data.  Assume that $\LL$ is $C^1$
on a neighbourhood of their common bounded state range.  Then the two
solutions coincide almost everywhere in $\QT$, and $\lambda=0$.
\end{theorem}

\begin{proof}
Fix $\alpha\ge\alpha_M$ from
Lemma~\ref{lem:relative-coercivity} and set
$$
 P(t)=\frac\alpha2\norm{\rho-\widehat\rho}_{L^2}^2
     +\frac\alpha2\norm{\eta-\widehat\eta}_{L^2}^2.
$$
The balance laws and Lemma~\ref{lem:time-representatives} imply
$\partial_t(\rho-\widehat\rho)\in L^2(0,T;H^{-1})$ and
$\partial_t(\eta-\widehat\eta)\in L^2(\QT)$; the standard
$H^1$--$L^2$--$H^{-1}$ chain rule therefore shows that $P$ is absolutely
continuous.  Put
$$
 V=(\nabla(\rho-\widehat\rho),0,\eta-\widehat\eta),
 \qquad
 \Delta J=\LL(U)Z-\LL(\widehat U)\widehat Z,
 \qquad
 K_{\widehat Z}=\norm{\widehat Z}_{L^\infty(\QT)}.
$$
Subtracting the two phase balances and integrating by parts in space gives
for almost every time
\begin{align}
 P'(t)
 &=-\alpha\int_\Omega V\cdot\Delta J\dx
 =-\alpha\int_\Omega
 V\cdot\LL(U)(Z-\widehat Z)\dx
 -\alpha\int_\Omega
 V\cdot(\LL(U)-\LL(\widehat U))\widehat Z\dx.
 \label{eq:penalty-derivative}
\end{align}
Local Lipschitz continuity of $\LL$, boundedness of $\widehat Z$, Young's
inequality, and \eqref{eq:statevector-difference-controlled} imply, for every
$\delta>0$,
\begin{equation}
 |P'(t)|
 \le\delta\norm{Z-\widehat Z}_{L^2}^2
 +C_{\delta,M,\alpha,K_{\widehat Z}}\mathcal H_\alpha(t).
 \label{eq:penalty-bound}
\end{equation}

The mobility remainder in
\eqref{eq:relative-exergy-inequality} satisfies
\begin{align}
 \left|(Z-\widehat Z)^\top
 (\LL(\widehat U)-\LL(U))\widehat Z\right|
 \le\delta|Z-\widehat Z|^2
 +C_{\delta,M,K_{\widehat Z}}|U-\widehat U|^2.
 \label{eq:mobility-remainder-bound}
\end{align}
Likewise, \eqref{eq:entropy-variable-remainder} gives
\begin{equation}
 |\mathcal R_Y(U\mid\widehat U)
       \cdot\partial_t\widehat U|
 \le C_M\norm{\partial_t\widehat U}_{L^\infty}
       |U-\widehat U|^2.
 \label{eq:entropy-remainder-bound}
\end{equation}
Finally, positivity of $\lambda$ and $\widehat\theta$ yields
\begin{equation}
 \pair{\lambda_s}{\partial_t\widehat\theta(s)}
 \le\norm{\partial_t\ln\widehat\theta(s)}_{L^\infty}
       \pair{\lambda_s}{\widehat\theta(s)}.
 \label{eq:defect-relative-bound}
\end{equation}

Add the integral of \eqref{eq:penalty-bound} to
\eqref{eq:relative-exergy-inequality}, choose $\delta$ small relative to
$\lambda_0$, and define, for almost every $t$,
$$
 \Phi(t)=\mathcal G(t)
 +\frac{\alpha}{2}\norm{\rho-\widehat\rho}_{L^2}^2
 +\frac{\alpha}{2}\norm{\eta-\widehat\eta}_{L^2}^2.
$$
Equivalently,
$\Phi(t)=\mathcal H_\alpha(t)
+\pair{\lambda_t}{\widehat\theta(t)}$ for almost every $t$.
Equations \eqref{eq:statevector-difference-controlled}--
\eqref{eq:defect-relative-bound} give
\begin{equation}
 \Phi(t)+c\int_0^t\norm{Z-\widehat Z}_{L^2}^2\,\mathrm ds
 \le \Phi(0)+\int_0^tG(s)\Phi(s)\,\mathrm ds,
 \label{eq:relative-gronwall}
\end{equation}
where
$$
 G(s)=C_*\Bigl(1+\norm{\partial_t\widehat U(s)}_{L^\infty}
 +\norm{\partial_t\ln\widehat\theta(s)}_{L^\infty}\Bigr)
 \in L^1(0,T),
$$
and $C_*$ depends only on the compact state range, the fixed penalty
parameter $\alpha$, the local $C^1$ bound for $\LL$, and
$K_{\widehat Z}$.
The common initial data and the defect-free initial trace give $\Phi(0)=0$.
Gronwall's lemma yields $\Phi=0$.  Coercivity then gives
$z=\widehat z$ and $\theta=\widehat\theta$ almost everywhere.  The two
constitutive relations imply equality of the chemical potentials.  Finally,
$$
 0=\pair{\lambda_t}{\widehat\theta(t)}
 \ge M^{-1}\lambda_t(\overline\Omega)
$$
for almost every $t$, and hence $\lambda=0$.
\end{proof}

\begin{remark}
The bounded-range assumption on the weak solution is a localisation
hypothesis used in the present proof.  On a common compact thermodynamic
range it yields the local coercivity of the penalised relative entropy, the
quadratic Taylor estimate for the entropy variables, and the local
Lipschitz estimate for the state-dependent Onsager operator.

A possible route to removing the pointwise boundedness of the weak solution
is an essential--residual decomposition relative to the compact range of
the strong solution.  This strategy is standard in relative-energy
arguments for compressible thermodynamic systems; see, for example,
\cite[Sect.~5.4.1 and the proof of Thm.~6.2, in particular
(6.26)--(6.28)]{FeireislLukacovaMizerovaShe2021}.  On the essential region
one uses the same local quadratic estimates as above, whereas on the
residual region the nonlinear remainders have to be controlled directly by
the coercive growth of the relative energy.

For the constitutive law considered here, such a global residual estimate
requires additional work.  Indeed, the bulk relative entropy associated
with a comparison inverse temperature $\widehat\theta$ contains the phase
contribution
$\widehat\theta F_1(z)-F_0(z)$,
whereas Assumption~\ref{ass:potential} guarantees quartic coercivity only
for $aF_1-F_0$ with one fixed coefficient $a$.  Moreover, the phase
components of the entropy variable satisfy
$-D_z\psi(z,\theta)=DF_0(z)-\theta DF_1(z)$,
and therefore exhibit mixed growth of order
$(1+\theta)(1+|z|^3)$.  Establishing an essential--residual relative-entropy
estimate for these terms would require exploiting the precise constitutive
structure rather than the local Taylor bounds used above.  We therefore
retain the bounded-range assumption in the present theorem.  Its removal
is a natural extension of the analysis.
\end{remark}

\subsection{Long-time behaviour and stationary states}
\label{sec:long-time}

We conclude with a long-time consequence of the entropy--availability
structure.  The basic mechanism goes back to the Lyapunov analysis of
Alt and Pawlow; compare \cite[Proposition~2.6]{alt1991mathematical} and
\cite[Section~3.4]{alt1992mathematical}.  Their convergence-to-one-state
results use an additional isolation and local stability hypothesis.  At the
energy regularity considered here, the unconditional conclusion is the
existence of stationary $\omega$-limit states.

By a global generalised dissipative weak solution we mean a sextuple
whose restriction to every finite interval $(0,T)$ is a solution in the
sense of Definition~\ref{def:weak-solution}, with compatible time
representatives on overlapping intervals.

\begin{corollary}[Existence of global trajectories]
\label{cor:global-weak-trajectory}
Under the hypotheses of Theorem~\ref{thm:convergence}, there exists at
least one global generalised dissipative weak solution on $(0,\infty)$.
\end{corollary}

\begin{proof}
Choose any sequences $h_k\downarrow0$ and $\tau_k\downarrow0$.  For each
$k$, the one-step existence theorem may be iterated for all $n\ge0$, thereby
producing a positive discrete trajectory on $[0,\infty)$.  The state and
dissipation estimates are uniform in the terminal time, whereas the remaining
space--time bounds are uniform on each fixed bounded interval.  Apply the
compactness and limit arguments of Section~\ref{sec:existence} first on
$(0,1)$ and then successively on $(0,m)$, $m\in\mathbb N$, each time
extracting a subsequence of the preceding one.  The diagonal subsequence
converges on every bounded time interval.  Since the subsequences are nested,
the limits on overlapping intervals agree.  The same diagonal extraction is
performed for the critical-energy measures.  The weakly continuous time
representatives are then constructed from the global distributional balances
as in Lemma~\ref{lem:time-representatives}, and their restrictions are
compatible.  The resulting family is a global solution.
\end{proof}

\begin{definition}[Stationary state]
\label{def:stationary-state}
A stationary state is a quadruple
$(\rho_\infty,\eta_\infty,\theta_\infty,\mu_{\rho,\infty})$ with
$$
 \rho_\infty,\eta_\infty\in H^1(\Omega),
 \qquad \theta_\infty>0,
 \qquad \mu_{\rho,\infty}\in\mathbb R,
$$
such that, with $z_\infty=(\rho_\infty,\eta_\infty)$,
\begin{align}
 \gamma_\rho(\nabla\rho_\infty,\nabla v)
 +(\partial_\rho\psi(z_\infty,\theta_\infty),v)
 &=\mu_{\rho,\infty}\int_\Omega v\,\mathrm dx,
 &&v\in H^1(\Omega),
 \label{eq:stationary-rho}\\
 \gamma_\eta(\nabla\eta_\infty,\nabla w)
 +(\partial_\eta\psi(z_\infty,\theta_\infty),w)
 &=0,
 &&w\in H^1(\Omega),
 \label{eq:stationary-eta}\\
 \int_\Omega\rho_\infty\,\mathrm dx
 &=\int_\Omega\rho_0\,\mathrm dx.
 \label{eq:stationary-mass}
\end{align}
The inverse temperature is spatially constant, so
$Z_\infty=(\nabla\mu_{\rho,\infty},-\nabla\theta_\infty,0)=0$ and all three
Onsager fluxes vanish.
\end{definition}

\begin{lemma}[Finite total dissipation and uniform orbit bounds]
\label{lem:long-time-dissipation}
Let $(\rho,\mu_\rho,\theta,\eta,\mu_\eta,\lambda)$ be a global generalised
dissipative weak solution.  Then
\begin{equation}
 \int_0^\infty\!\!\int_\Omega
 Z^\top\LL(z,\theta)Z\,\mathrm dx\,\mathrm dt<\infty,
 \label{eq:finite-total-dissipation}
\end{equation}
and, for almost every $t>0$,
\begin{align}
 &\|z(t)\|_{H^1}^2+\|z(t)\|_{L^4}^4
 +\|\theta(t)\|_{L^1}
 +\|\theta(t)^{-q}\|_{L^1}
 +\|\lambda_t\|_{\mathcal M(\overline\Omega)}\le C.
 \label{eq:uniform-long-time-orbit}
\end{align}
In particular,
\begin{equation}
 \int_0^\infty
 \bigl(
 \|\nabla\mu_\rho\|_{L^2}^2
 +\|\nabla\theta\|_{L^2}^2
 +\|\mu_\eta\|_{L^2}^2
 \bigr)\,\mathrm dt<\infty.
 \label{eq:finite-component-dissipation}
\end{equation}
\end{lemma}

\begin{proof}
Let
$$
 E_0:=\int_\Omega e(z_0,\theta_0)\,\mathrm dx
$$
and define, for almost every $t>0$, the defect-aware availability
$$
 \mathfrak A_a(t)
 :=a\mathfrak E_t(\overline\Omega)-\Stot(z(t),\theta(t)).
$$
Energy conservation gives $\mathfrak E_t(\overline\Omega)=E_0$.  Moreover,
using \eqref{eq:total-energy-measure},
\begin{align*}
 \mathfrak A_a(t)
 ={}&\int_\Omega
 \bigl(ae(z(t),\theta(t))-\Sbulk(z(t),\theta(t))\bigr)\,\mathrm dx
 +\frac{\gamma_\rho}{2}\|\nabla\rho(t)\|_{L^2}^2
 +\frac{\gamma_\eta}{2}\|\nabla\eta(t)\|_{L^2}^2
 +a\lambda_t(\overline\Omega).
\end{align*}
The entropy inequality therefore yields
\begin{equation}
 \mathfrak A_a(t)
 +\int_0^t\!\!\int_\Omega Z^\top\LL(z,\theta)Z
 \,\mathrm dx\,\mathrm ds
 \le \mathfrak A_a(0)
 \quad\text{for almost every }t>0.
 \label{eq:global-availability-dissipation}
\end{equation}
By Lemma~\ref{lem:availability-coercivity}, there are constants $c,C>0$
such that, for almost every $t>0$,
$$
 \mathfrak A_a(t)\ge
 c\Bigl(\|z(t)\|_{H^1}^2+\|z(t)\|_{L^4}^4
 +\|\theta(t)\|_{L^1}+\|\theta(t)^{-q}\|_{L^1}
 +\|\lambda_t\|_{\mathcal M(\overline\Omega)}\Bigr)-C.
$$
Since \eqref{eq:global-availability-dissipation} also gives
$\mathfrak A_a(t)\le\mathfrak A_a(0)$, this proves
\eqref{eq:uniform-long-time-orbit}.  Next choose admissible times
$s_k\to\infty$ for which \eqref{eq:global-availability-dissipation}
holds.  The preceding lower bound yields
$$
 \int_0^{s_k}\!\!\int_\Omega Z^\top\LL Z\,\mathrm dx\,\mathrm dt
 \le \mathfrak A_a(0)+C.
$$
Monotone convergence as $k\to\infty$ proves
\eqref{eq:finite-total-dissipation}.  Finally,
Assumption~\ref{ass:Onsager} gives
$$
 Z^\top\LL Z\ge\lambda_0
 \bigl(|\nabla\mu_\rho|^2+|\nabla\theta|^2+|\mu_\eta|^2\bigr),
$$
which proves \eqref{eq:finite-component-dissipation}.
\end{proof}

\begin{theorem}[Stationary $\omega$-limit points]
\label{thm:stationary-omega-limit}
Let $(\rho,\mu_\rho,\theta,\eta,\mu_\eta,\lambda)$ be a global generalised
dissipative weak solution.  Then there exist times $t_j\to\infty$, a
stationary state
$(\rho_\infty,\eta_\infty,\theta_\infty,\mu_{\rho,\infty})$, and a measure
$\Lambda_\infty\in\mathcal M_+(\overline\Omega)$ such that, after
extraction,
\begin{align}
 \rho(t_j)&\to\rho_\infty
 &&\text{strongly in }H^1(\Omega),
 \label{eq:long-time-rho}\\
 \eta(t_j)&\to\eta_\infty
 &&\text{strongly in }H^1(\Omega),
 \label{eq:long-time-eta}\\
 \theta(t_j)&\to\theta_\infty
 &&\text{strongly in }H^1(\Omega),
 \label{eq:long-time-theta}\\
 \mu_\rho(t_j)&\to\mu_{\rho,\infty}
 &&\text{strongly in }H^1(\Omega),
 \label{eq:long-time-murho}\\
 \mu_\eta(t_j)&\to0
 &&\text{strongly in }L^2(\Omega).
 \label{eq:long-time-mueta}
\end{align}
Here $\theta_\infty$ is a positive spatial constant.  In addition, the
total critical-energy measures satisfy
\begin{equation}
 \eps_q\theta(t_j)^{-q}\,\mathrm dx+\lambda_{t_j}
 \stackrel{*}{\rightharpoonup}
 \eps_q\theta_\infty^{-q}\,\mathrm dx+\Lambda_\infty
 \quad\text{in }\mathcal M(\overline\Omega),
 \label{eq:long-time-critical-measure}
\end{equation}
and
\begin{equation}
 \int_\Omega e(z_\infty,\theta_\infty)\,\mathrm dx
 +\Lambda_\infty(\overline\Omega)
 =\int_\Omega e(z_0,\theta_0)\,\mathrm dx.
 \label{eq:stationary-energy-with-defect}
\end{equation}
The measure $\Lambda_\infty$ includes both any pre-existing defect
$\lambda_{t_j}$ and any additional concentration of
$\eps_q\theta(t_j)^{-q}$ along the long-time sequence.
\end{theorem}

\begin{proof}
Set
$$
 \mathcal D(t):=\int_\Omega Z(t)^\top\LL(z(t),\theta(t))Z(t)\,\mathrm dx.
$$
All pointwise-in-time statements below are valid on one common full-measure
set: the constitutive relations follow there from
\eqref{eq:weak-murho}--\eqref{eq:weak-mueta} by testing with a countable
dense subset of $H^1(\Omega)$, while the energy decomposition, the uniform
orbit bounds, and the entropy dissipation hold there by
Definition~\ref{def:weak-solution} and
Lemma~\ref{lem:long-time-dissipation}.  Since $\mathcal D\in L^1(0,\infty)$,
$$
 \int_j^{j+1}\mathcal D(t)\,\mathrm dt\longrightarrow0.
$$
We may therefore choose
$t_j\in(j,j+1)$ in the above full-measure set such that
$$
 \mathcal D(t_j)
 \le 2\int_j^{j+1}\mathcal D(t)\,\mathrm dt.
$$
Consequently,
\begin{equation}
 \mathcal D(t_j)\longrightarrow0.
 \label{eq:selected-dissipation-zero}
\end{equation}
For brevity, write
$$
 z_j=z(t_j),\quad \theta_j=\theta(t_j),\quad
 \mu_{\rho,j}=\mu_\rho(t_j),\quad
 \mu_{\eta,j}=\mu_\eta(t_j),\quad
 \lambda_j=\lambda_{t_j}.
$$
Assumption~\ref{ass:Onsager} and
\eqref{eq:selected-dissipation-zero} imply
\begin{equation}
 \|\nabla\mu_{\rho,j}\|_{L^2}
 +\|\nabla\theta_j\|_{L^2}
 +\|\mu_{\eta,j}\|_{L^2}\longrightarrow0.
 \label{eq:selected-forces-zero}
\end{equation}

We first identify the inverse-temperature limit.  Let
$\overline\theta_j=|\Omega|^{-1}\int_\Omega\theta_j\,\mathrm dx$.
The uniform $L^1$ bound gives $\overline\theta_j\le C$.  Since
$s\mapsto s^{-q}$ is convex, Jensen's inequality and
\eqref{eq:uniform-long-time-orbit} give
$$
 (\overline\theta_j)^{-q}
 \le |\Omega|^{-1}\int_\Omega\theta_j^{-q}\,\mathrm dx\le C,
$$
so $\overline\theta_j\ge c>0$.  After extraction,
$\overline\theta_j\to\theta_\infty$ for some $\theta_\infty>0$.
The mean-zero Poincar\'e inequality and
\eqref{eq:selected-forces-zero} then yield
\begin{equation}
 \|\theta_j-\theta_\infty\|_{H^1}
 \le C\|\nabla\theta_j\|_{L^2}
   +|\Omega|^{1/2}|\overline\theta_j-\theta_\infty|
 \longrightarrow0,
 \label{eq:theta-snapshot-H1}
\end{equation}
which proves \eqref{eq:long-time-theta}.

The uniform orbit bounds and Rellich's theorem give, after a further
extraction,
\begin{equation}
 z_j\rightharpoonup z_\infty\quad\text{in }H^1(\Omega)^2,
 \qquad
 z_j\to z_\infty\quad\text{in }L^p(\Omega)^2
 \quad(1\le p<6).
 \label{eq:z-snapshot-precompactness}
\end{equation}
The generalised Poincar\'e inequality and the uniform $L^1$ bound show that
$\theta_j$ is uniformly bounded in $L^2$.  Testing the $\rho$ constitutive
relation by $1$ gives
$$
 \overline\mu_{\rho,j}
 :=|\Omega|^{-1}\int_\Omega\mu_{\rho,j}\,\mathrm dx
 =|\Omega|^{-1}\int_\Omega\partial_\rho\psi(z_j,\theta_j)\,\mathrm dx.
$$
Using \eqref{eq:phase-growth}, the $H^1\hookrightarrow L^6$ bound for
$z_j$, and the $L^2$ bound for $\theta_j$, we obtain
\begin{align*}
 |\overline\mu_{\rho,j}|
 &\le C\left(1+\|z_j\|_{L^3}^3
 +\|\theta_j\|_{L^2}
   \bigl(1+\|z_j\|_{L^6}^3\bigr)\right)\le C.
\end{align*}
Thus, after extraction,
$\overline\mu_{\rho,j}\to\mu_{\rho,\infty}\in\mathbb R$.
Poincar\'e's inequality and \eqref{eq:selected-forces-zero} give
\begin{equation}
 \mu_{\rho,j}\to\mu_{\rho,\infty}
 \quad\text{strongly in }H^1(\Omega),
 \qquad
 \mu_{\eta,j}\to0
 \quad\text{strongly in }L^2(\Omega).
 \label{eq:chemical-snapshot-convergence}
\end{equation}

For $\sigma\in\{\rho,\eta\}$, the explicit structure
$\partial_\sigma\psi=-\partial_\sigma F_0
 +\theta\partial_\sigma F_1$, the polynomial growth, and
\eqref{eq:theta-snapshot-H1}--\eqref{eq:z-snapshot-precompactness} imply
\begin{equation}
 \partial_\sigma\psi(z_j,\theta_j)
 \longrightarrow
 \partial_\sigma\psi(z_\infty,\theta_\infty)
 \quad\text{strongly in }L^{6/5}(\Omega).
 \label{eq:stationary-nonlinearity-convergence}
\end{equation}
Indeed,
$$
 |DF_k(z_j)-DF_k(z_\infty)|
 \le C(1+|z_j|^2+|z_\infty|^2)|z_j-z_\infty|,
$$
so the difference converges in $L^{6/5}$ by the boundedness of the
quadratic factor in $L^3$ and the strong $L^2$ convergence of $z_j$.
Furthermore,
$$
 \|(\theta_j-\theta_\infty)DF_1(z_j)\|_{L^{6/5}}
 \le\|\theta_j-\theta_\infty\|_{L^6}
      \|DF_1(z_j)\|_{L^{3/2}}\longrightarrow0.
$$

Passing to the limit in the two constitutive relations at the times $t_j$
now gives \eqref{eq:stationary-rho}--\eqref{eq:stationary-eta}.  To improve
\eqref{eq:z-snapshot-precompactness} to strong $H^1$ convergence, subtract
the limiting $\rho$ relation from the relation at $t_j$ and test by
$\rho_j-\rho_\infty$.  We obtain
\begin{align*}
 \gamma_\rho\|\nabla(\rho_j-\rho_\infty)\|_{L^2}^2
 ={}&(\mu_{\rho,j}-\mu_{\rho,\infty},
       \rho_j-\rho_\infty)-(\partial_\rho\psi(z_j,\theta_j)
   -\partial_\rho\psi(z_\infty,\theta_\infty),
       \rho_j-\rho_\infty).
\end{align*}
The first term tends to zero by
\eqref{eq:chemical-snapshot-convergence}; the second tends to zero by
\eqref{eq:stationary-nonlinearity-convergence} and the uniform $L^6$ bound
for $\rho_j-\rho_\infty$.  The same argument for $\eta$, using
$\mu_{\eta,j}\to0$, yields
$$
 z_j\to z_\infty\quad\text{strongly in }H^1(\Omega)^2.
$$
This proves \eqref{eq:long-time-rho}--\eqref{eq:long-time-eta};
\eqref{eq:long-time-murho}--\eqref{eq:long-time-mueta} follow from
\eqref{eq:chemical-snapshot-convergence}.  Mass conservation and strong
$L^2$ convergence give \eqref{eq:stationary-mass}.

It remains to identify the limiting energy.  Define the non-negative
measures
$$
 \Xi_j:=\eps_q\theta_j^{-q}\,\mathrm dx+\lambda_j.
$$
They have uniformly bounded mass by
\eqref{eq:uniform-long-time-orbit}; hence, after extraction,
$\Xi_j\stackrel{*}{\rightharpoonup}\Xi_\infty$ in
$\mathcal M_+(\overline\Omega)$.  Since
$\theta_j\to\theta_\infty$ almost everywhere, Fatou's lemma gives, for
every non-negative $\phi\in C(\overline\Omega)$,
$$
 \int_\Omega\eps_q\theta_\infty^{-q}\phi\,\mathrm dx
 \le\liminf_{j\to\infty}
 \int_\Omega\eps_q\theta_j^{-q}\phi\,\mathrm dx
 \le\int_{\overline\Omega}\phi\,\mathrm d\Xi_\infty.
$$
Consequently,
$$
 \Lambda_\infty
 :=\Xi_\infty-\eps_q\theta_\infty^{-q}\,\mathrm dx
 \in\mathcal M_+(\overline\Omega),
$$
and \eqref{eq:long-time-critical-measure} follows.

Finally, write $e=e_\reg+\eps_q\theta^{-q}$ as in
\eqref{eq:energy-split}.  The strong $H^1$ convergence of $z_j$ gives
$F_1(z_j)\to F_1(z_\infty)$ in $L^1$.  Moreover,
$\theta_j^{-1}$ is bounded in $L^q$ and converges almost everywhere, so it
converges strongly in $L^1$ by Vitali's theorem.  The positive logarithmic
part is one-Lipschitz in $\theta$, while
$[(-\ln s)^+]^p\le C_{p,q}(1+s^{-q})$ for some $p>1$; hence both logarithmic
parts converge strongly in $L^1$.  Therefore
\begin{equation}
 e_\reg(z_j,\theta_j)\to
 e_\reg(z_\infty,\theta_\infty)
 \quad\text{strongly in }L^1(\Omega).
 \label{eq:long-time-regular-energy}
\end{equation}
At the selected times, energy conservation reads
$$
 \int_\Omega e_\reg(z_j,\theta_j)\,\mathrm dx
 +\Xi_j(\overline\Omega)=E_0.
$$
Pass to the limit using \eqref{eq:long-time-critical-measure} and
\eqref{eq:long-time-regular-energy} to obtain
\eqref{eq:stationary-energy-with-defect}.
\end{proof}

\begin{remark}
 Under stronger regularity assumptions, it should also be possible to estimate the convergence speed towards the limit solutions by using Łojasiewicz–Simon inequalities \cite{CHILL2003572} similarly to the Cahn-Hilliard case \cite{Rybka1999}.   
\end{remark}

\section{Summary and Outlook}

In this work, we have developed an existence, stability, and long-time
theory for a thermodynamically consistent non-isothermal phase-field system
with two order parameters and a fully non-diagonal Onsager coupling.  The
main analytical difficulty is the simultaneous treatment of positivity of
the inverse temperature, the nonlinear thermal constitutive relations, and
the weak compactness available for the internal energy.  To address this,
we introduced a dimension-adapted singular thermal contribution and used
the availability as the principal coercive functional.  At the discrete
level, this structure yields exact conservation of mass and internal
energy, a discrete entropy inequality, and uniform estimates compatible
with continuous thermodynamics.  The inverse-moment bound gives a
strictly positive inverse temperature on every fixed finite-element mesh,
with the threshold $q\ge d$ being sharp for the corresponding nodal
barrier argument.  Together with a Brouwer degree construction, this
provides global discrete trajectories.  Compactness can then be obtained
for arbitrary $h,\tau\to0$, without any coupling condition between the
space and time discretisation parameters, and leads to the global
generalised dissipative weak solutions of
Theorem~\ref{thm:convergence}.

The resulting solution concept isolates precisely the possible loss of
compactness in the thermal energy.  All lower-order thermal contributions
are strongly compact, whereas the critical term
$\eps_q\theta^{-q}$ may concentrate and is therefore represented by the
non-negative defect measure $\lambda$.  In this sense, the defect is
confined to the single constitutive quantity for which the available
bounds are critical.  Proposition~\ref{prop:no-defect} shows that it
disappears as soon as additional equi-integrability of the singular energy
is available.  Moreover, the relative-entropy inequality shows that the
generalised solution concept is compatible with the strong theory:
within a common bounded thermodynamic state range, a generalised
dissipative weak solution coincides with a sufficiently regular strong
solution emanating from the same initial data, and the defect measure
vanishes; see Theorem~\ref{thm:bounded-range-wsu}.

The entropy--availability structure also provides information beyond
finite-time existence.  Global trajectories have finite total Onsager
dissipation and remain uniformly bounded in the natural energy variables.
Consequently, every global generalised dissipative weak solution possesses
a sequence $t_j\to\infty$ along which the phase variables and inverse
temperature converge strongly in $H^1(\Omega)$ to a stationary state,
while the chemical potentials converge to their equilibrium values; see
Theorem~\ref{thm:stationary-omega-limit}.  The limiting inverse temperature
is a positive spatial constant.  At the level of the critical thermal
energy, a non-negative asymptotic defect may remain, accounting both for concentration already present in $\lambda_{t_j}$ and for concentration
generated along the selected long-time sequence. Thus, the present
argument identifies stationary $\omega$-limit states but does not in
general assert convergence of the whole trajectory to a single
equilibrium.

Several questions remain open.  First, the parameters $\eps_1$ and
$\eps_q$ in the modified free energy are fixed throughout the present
analysis.  Passing to the unmodified constitutive model by letting
$\eps_1,\eps_q\downarrow0$ would require uniform estimates in these
parameters and separate identification of singular thermal
contributions.  Alternatively, comparable thermal coercivity may be
generated through sufficiently strong temperature-dependent heat
conductivities.  Second, the bounded-range hypothesis in the
weak--strong uniqueness theorem is used to obtain uniform local
coercivity and quadratic control of the nonlinear remainders.  Removing
this assumption, for instance, through an essential--residual
decomposition adapted to the availability, would substantially strengthen
the stability theory.  Finally, a natural long-time question is whether
additional assumptions on the set of stationary states allow one to
upgrade subsequential convergence to convergence of the entire trajectory
and to eliminate the asymptotic defect.  Under suitable regularity and
isolation hypotheses, a possible approach is provided by
Lojasiewicz--Simon inequalities, which may also yield quantitative
convergence rates. 

\section*{Acknowledgment}
 \noindent A.B. has been supported by the Deutsche Forschungsgemeinschaft (DFG, German Research Foundation) by the SPP 2256 "Variational Methods for Predicting Complex Phenomena in Engineering Structures and Materials" under Project No. 441153493 and by the Mainz
Institute of Multiscale Modeling M3odel.

\bibliographystyle{elsarticle-num}
\bibliography{literature}

\end{document}